\RequirePackage[displaymath, mathlines]{lineno}
\documentclass[12pt]{article}
\usepackage{Sweave}

\usepackage{graphicx}        
\usepackage{multicol}        

\usepackage{amsfonts}
\usepackage[latin1]{inputenc}
\usepackage[T1]{fontenc}

\usepackage{amssymb}
\usepackage{amsthm}

\usepackage{bm}
\usepackage{bbold}
\usepackage{graphicx}
\usepackage{csquotes}
\usepackage{verbatim} 
\usepackage{eurosym}
\usepackage{siunitx}
\usepackage{hhline}
\usepackage{colortbl}
\usepackage[table]{xcolor}
\usepackage{multirow}
\usepackage{tikz-cd} 
\usepackage{csquotes}
\usepackage{caption}
\usepackage{subcaption}
\usepackage[framemethod=tikz]{mdframed}
\usepackage{array}
\usepackage{amsmath, amscd}
\usepackage{amsthm}

\newtheorem{theorem}{Theorem}[section]

\newtheorem{lemma}[theorem]{Lemma}

\newtheorem{example}{Example}[section]
\newtheorem{remark}{Remark}

\newcolumntype{R}[1]{>{\raggedleft\let\newline\\\arraybackslash\hspace{0pt}}m{#1}}

\newcommand{\clr}{\mbox{clr}}

\title{Compositional splines for representation of density functions}

\author{J. Machalov\'a$^{a}$, R. Talsk\'a$^{a\ast}$, K. Hron$^a$, A. G\'aba$^b$ \\
$^{a}$Department of Mathematical Analysis and Applications of Mathematics, \\ Faculty of Science, Palack\'y University Olomouc, \\ 17. listopadu 12, CZ-77146 Olomouc, Czech Republic \\ 
$^{b}$Department of Natural Sciences in Kinanthropology, \\ Faculty of Physical Culture, Palack\'y University Olomouc, \\ t\v r. M\'iru 117, CZ-77111 Olomouc, Czech Republic \\
$^\ast$Corresponding author. Email: talskarenata@seznam.cz
}

\begin{document} 

\maketitle

\begin{abstract}
In the context of functional data analysis, probability density functions as non-negative functions are characterized by specific properties of scale invariance and relative scale which enable to represent them with the unit integral constraint without loss of information. On the other hand, all these properties are a challenge when the densities need to be approximated with spline functions, including construction of the respective spline basis. The Bayes space methodology of density functions enables to express them as real functions in the standard $L^2$ space using the centered log-ratio transformation. The resulting functions satisfy the zero integral constraint. This is a key to propose a new spline basis, holding the same property, and consequently to build a new class of spline functions, called compositional splines, which can approximate probability density functions in a consistent way. The paper provides also construction of smoothing compositional splines and possible orthonormalization of the spline basis which might be useful in some applications. Finally, 
statistical processing of densities using the new approximation tool is demonstrated in case of simplicial functional principal component analysis with anthropometric data.
\end{abstract}

\section{Introduction}


Probability density functions are popularly known as non-negative functions satisfying the unit integral constraint. This clearly inhibits their direct
processing using standard methods of functional data analysis \cite{ramsay05} since unconstrained functions are assumed there. The same holds also for approximation of the raw input data using splines which is commonly considered to be a key step in functional data analysis. But more severely, in addition to the apparent unit integral constraint of densities which might seem to represent just a kind of numerical obstruction, density
functions are rather characterized by deeper geometrical properties that need to be taken into account for any reliable analysis
\cite{egozcue06,boogaart10,boogaart14}. Specifically, in contrast to functions in the standard $L^2$ space, densities obey the scale invariance and
relative scale properties \cite{hron16}. Scale invariance means that not just the representation of densities with the unit integral constraint, but any its positive multiple conveys the same information about relative contributions of Borel sets on the whole probability mass. Relative scale can be
explained directly with an example: the relative increase of a probability over a Borel set from 0.05 to 0.1 (2 multiple) differs from the increase
0.5 to 0.55 (1.1 multiple), although the absolute differences are the same in both cases. If we restrict to a bounded support
$I=[a,b]\subset\mathbb{R}$ that is mostly used in practical applications \cite{delicado11,hron16,menafoglio14,menafoglio16}, density functions can be
represented with respect to Lebesgue reference measure using the Bayes space $\mathcal{B}^2(I)$ of functions with square-integrable logarithm
\cite{egozcue06,boogaart14}.

The Bayes space $\mathcal{B}^2(I)$ has structure of separable Hilbert space that enables construction of an isometric isomorphism between
$\mathcal{B}^2(I)$ and $L^2(I)$, the $L^2$ space restricted to $I$. Accordingly, analogies of summing two functions and multiplication of a function
by a real scalar in the $L^2$ space together with an inner product between two densities are required. Given two absolutely integrable density
functions $f,g\in \mathcal{B}^2(I)$ and a real number $\alpha\in\mathbb{R}$ we indicate with $f\oplus g$ and $\alpha\odot f$ the \textit{perturbation} and
\textit{powering} operations, defined as
\begin{equation}
\label{operations}
(f\oplus g)(x)=\frac{f(x)g(x)}{\int_If(y)g(y)\,\mathrm{d}y},\quad (\alpha\odot f)(x)=\frac{f(x)^{\alpha}}{\int_If(y)^{\alpha}\,\mathrm{d}y}, \quad
x\in I,
\end{equation}
respectively. The resulting functions are readily seen to be probability density functions, though, note that the unit integral constraint representation was
chosen just for the sake of convenient interpretation. In \cite{egozcue06}, it is proven that $\mathcal{B}^2(I)$ endowed with the operations
$(\oplus,\odot)$ is a vector space. Note that the neutral elements of perturbation and powering are $e(x)=1/\eta$, with $\eta=b-a$ (i.e., the uniform
density), and 1, respectively. The difference between two elements $f,g\in \mathcal{B}^2(I)$, denoted by $f\ominus g$, is obtained as perturbation of $f$ with the reciprocal of $g$, i.e., $(f\ominus g)(x)=(f\oplus[(-1)\odot g])(x), x\in I$. Finally, to complete the Hilbert space structure, the
inner product is defined as
\begin{equation}\label{eq:inner-prod}
\langle f,g\rangle_\mathcal{B}=\frac{1}{2\eta}\int_I\int_I\ln\frac{f(x)}{f(y)}\ln\frac{g(x)}{g(y)}\,\mathrm{d}x\,\mathrm{d}y, \quad f,g \in \mathcal{B}^2(I).
\end{equation}
Form of the inner product clearly indicates that the relevant information in densities is contained in (log-)ratios between elements from the support $I$.

Density functions can be considered as functional counterparts to compositional data, positive vectors carrying relative information
\cite{aitchison86,pawlowsky15} that are driven by the Aitchison geometry \cite{pawlowsky01}. In order to enable their statistical processing using standard multivariate methods in real space \cite{eaton83}, the preferred strategy is to express them either in centered log-ratio (clr) coefficients \cite{aitchison86} with respect to a generating system, or in logratio coordinates, preferably with respect to an orthonormal basis \cite{egozcue03}. The latter coordinates (called also isometric log-ratio coordinates), as well as the clr coefficients, provide isometry between the Aitchison geometry and the real Euclidean space. A similar strategy is used also for densities in the Bayes space \cite{boogaart14}.
An isometric isomorphism between $\mathcal{B}^2(I)$ and $L^2(I)$ is represented by the \emph{centered log-ratio} (clr) transformation \cite{menafoglio14,boogaart14}, defined for $f\in\mathcal{B}^2(I)$ as
\begin{equation}
\label{fclr}
\mbox{clr}(f)(x) \equiv f_c(x)=\ln f(x)-\frac{1}{\eta}\int_I\ln f(y)\,\mathrm{d}y.
\end{equation}
We remark that such an isometry allows to compute operations and inner products among the
elements in $\mathcal{B}^2(I)$ in terms of their counterpart in $L^2(I)$ among the clr-transforms, i.e.
\[
\mbox{clr}(f\oplus g)(x)=f_c(x)+g_c(x),\ \mbox{clr}(\alpha\odot f)(x)=\alpha\cdot f_c(x)\]
and
\[
\langle f,g\rangle_\mathcal{B}=\langle f_c,g_c\rangle_2= \int_I f_c(x)g_c(x)\,\mathrm{d}x.\]
However, the clr transformation induces an additional constraint,
\begin{equation}
\label{clrzero}
\int_I\mbox{clr}(f)(x) \mathrm{d}x = \int_I \ln f(x)\, \mathrm{d}x - \int_I \frac{1}{\eta}\int_I\ln f(y)\,\mathrm{d}y\, \mathrm{d}x = 0,
\end{equation}
that needs to be taken into account for computation and analysis on clr-transformed density functions. As the clr space is clearly a subspace of $L^2(I)$, hereafter it is denoted as $L_0^2(I)$. The inverse clr transformation is obtained as
\begin{equation}\label{fclrinv}
    \mbox{clr}^{-1}[f_c](x)=\frac{\exp(f_c(x))}{\int_I\exp(f_c(y))\,\mbox{d}y};
\end{equation}
again as before, the denominator is used just to achieve the unit integral constraint representation of the resulting density (without loss of relative information, carried by the density function).

According to \cite{boogaart14}, it is not necessary to restrict ourselves to the constrained clr space, because a basis in $\mathcal{B}^2(I)$ can be
easily constructed. Specifically, let $\psi_0(x),\psi_1(x),\psi_2(x),\dots$ is a basis in $L^2(I)$ and assume that $\psi_0(x)$ is a constant function,
then $\varphi_1(x):=\exp(\psi_1(x)),$ $\varphi_2(x):=\exp(\psi_2(x)),\dots$ form a basis in $\mathcal{B}^2(I)$. Of course, also here an orthonormal basis
is preferable, but it is not always possible in applications. Nevertheless, if this would be so, then a function
$f\in\mathcal{B}^2(I)$ can be projected orthogonally to the space spanned, e.g., by the first $r$ functions
$\varphi_1(x),\varphi_2(x),\dots,\varphi_r(x)$. This is done through the respective coefficients $c_1,\dots,c_r$ in the basis expansion
\begin{equation}
\label{bexpansion}
f(x)=c_1\odot\varphi_1(x)\oplus c_2\odot\varphi_2(x)\oplus\ldots\oplus c_r\odot\varphi_r(x)\oplus
\ldots=\bigoplus_{i=1}^{\infty}c_i\odot\varphi_i(x),\ x\in I.
\end{equation}

Functional data analysis relies strongly on approximation of the input functions using splines \cite{ramsay05}. However, splines are mostly utilized
purely as an approximation tool, without considering further methodological consequences. Because statistical processing of density functions requires a deeper geometrical background, provided by the Bayes spaces, this should be followed also by the respective spline representation, performed
preferably in the clr space $L_0^2(I)$. In \cite{Mach}, a first attempt of constructing a spline representation that would honor the zero integral
constraint (\ref{clrzero}) was performed. The problem is that $B$-splines that form basis for the spline expansion in \cite{Mach} come from
$L^2(I)$, but not from $L_0^2(I)$. This paper presents an important step ahead -- such splines are constructed that form basis functions in the clr space $L_0^2(I)$. Consequently, the splines can be expressed also directly in $\mathcal{B}^2(I)$ and the spline representation formulated in terms of the
Bayes space which can be used for interpretation purposes; hereafter we refer to \textit{compositional splines}. Apart from methodological advantages, using compositional splines simplifies
construction and interpretation of spline coefficients that can be considered as coefficients of a (possibly orthonormal) basis in $\mathcal{B}^2(I)$.


The paper is organized as follows. In the next section the construction of splines basis in $L_0^2(I)$ is presented together with a comparison to
spline functions introduced in \cite{Mach}. Section 3 is devoted to smoothing splines in $L_0^2(I)$ and Section 4 discusses orthogonalization of basis functions
(that form, by construction, an oblique basis). Section 5 introduces a new class of splines that reflect the Bayes spaces methodology, compositional splines. Section 6 demonstrates usefulness of the new approximation tool in context of simplicial functional principal component analysis with anthropometric data and the final Section 7 concludes.


\section{Construction of spline in $L_0^2(I)$
}


Because the clr transformation enables to process density functions in the standard $L^2$ space, just restricted according to zero integral
constraint (\ref{clrzero}), it is natural that also construction of compositional splines should start in $L_0^2(I)$. Nevertheless, before doing so,
some basic facts about $B$-spline representation of splines are recalled, see 
\cite{deboor78,dierckx93,Schum07} for details. 
Let the sequence of knots $$\lambda_{0}=a<\lambda_{1}<\cdots<\lambda_{g}<b=\lambda_{g+1}$$
be given. The (normalized) $B$-spline of degree $0$ (order $1$) is defined as
$$
B_{i}^1(x) \, = \, \left\{
\begin{array}{rl}
    1 &\quad\mbox{if}\; x\in[\lambda_{i},\lambda_{i+1})\\
    0 &\quad\mbox{otherwise}
\end{array}\right.
$$
and the (normalized) $B$-spline of degree $k$, $k\in\mathbb{N}$, (order $k+1$) is defined by
$$
B_{i}^{k+1}(x) \, = \,\frac{x-\lambda_i}{\lambda_{i+k}-\lambda_{i}}B_i^k(x)+\frac{\lambda_{i+k+1}-x}{\lambda_{i+k+1}-\lambda_{i+1}}B_{i+1}^k(x).
$$
Now let the functions $Z_{i}^{k+1}(x)$ for $k\geq 0$, $k\in\mathbb{N}$, be defined
\begin{equation}\label{zb}
Z_{i}^{k+1}(x):=\frac{\mbox{d}}{\mbox{d}x} B_{i}^{k+2}(x),
\end{equation}
i.e., more precisely for $k=0$
$$
Z_{i}^1(x)=\left\{
\begin{array}{rl}
     \dfrac{1}{\lambda_{i+1}-\lambda_i} & \quad\mbox{if}\; x\in[\lambda_{i},\lambda_{i+1})\\
     & \\
     \dfrac{-1}{\lambda_{i+2}-\lambda_{i+1}} & \quad\mbox{if}\; x\in(\lambda_{i+1},\lambda_{i+2}]
\end{array}\right.
$$
and for $k\geq 1$
\begin{equation}
\label{defZ}
Z_i^{k+1}(x)=(k+1)\left(\frac{B_i^{k+1}(x)}{\lambda_{i+k+1}-\lambda_i}-\frac{B_{i+1}^{k+1}(x)}{\lambda_{i+k+2}-\lambda_{i+1}}\right).
\end{equation}

\bigskip
Noteworthy, functions $Z_i^{k+1}(x)$ have similar properties as $B$-splines $B_i^{k+1}(x)$.

\begin{enumerate}
    \item They are piecewise polynomials of degree $k$. Particularly, $Z_i^1(x)$ is a piecewise constant polynomial, $Z_i^2(x)$ is a piecewise
linear polynomial, see Figure \ref{o1}, $Z_i^3(x)$ is a piecewise quadratic polynomial, see Figure \ref{o2}. For other examples see Figures
\ref{o3}-\ref{o5}.
    \item It is evident that for $k\geq 1$ the function $Z_i^{k+1}(x)$ and its derivatives up to order $k-1$ are all continuous.
    \item It is easy to check that for $k\geq 0 $
    $$\mbox{supp}\;Z_i^{k+1}(x) \; = \; \mbox{supp}\;B_i^{k+2}(x)=[\lambda_i,\lambda_{i+k+2}],$$
    \noindent and of course
    $$Z_i^{k+1}(x) \; = \; 0 \quad \mbox{if} \quad x\notin[\lambda_i,\lambda_{i+k+2}].$$
    \item  From the perspective of $L_0^2(I)$, a crucial point is that the integral of $Z_i^{k+1}(x)$ equals to zero. If we consider Curry-Schoenberg $B$-spline
    $M_i^{k+1}(x)$ \cite{deboor78}, which are defined as
    $$M_i^{k+1}(x) \, := \, \dfrac{k+1}{\lambda_{i+k+1}-\lambda_i} \, B_i^{k+1}(x)$$
    with property
    $$\int\limits_{\mathbb{R}} M_i^{k+1}(x) \, \mbox{d}x \, = \, 1,$$
    than it is clear that
    \begin{equation}\label{ZM}
          Z_i^{k+1}(x) \, = \, M_i^{k+1}(x) - M_{i+1}^{k+1}(x)
    \end{equation}
    and
    $$\int\limits_{\mathbb{R}} Z_i^{k+1}(x) \, \mbox{d}x \, = \, 0.$$

\end{enumerate}


\begin{figure}[!ht]
  \begin{center}
    \includegraphics[width=6cm,height=4cm]{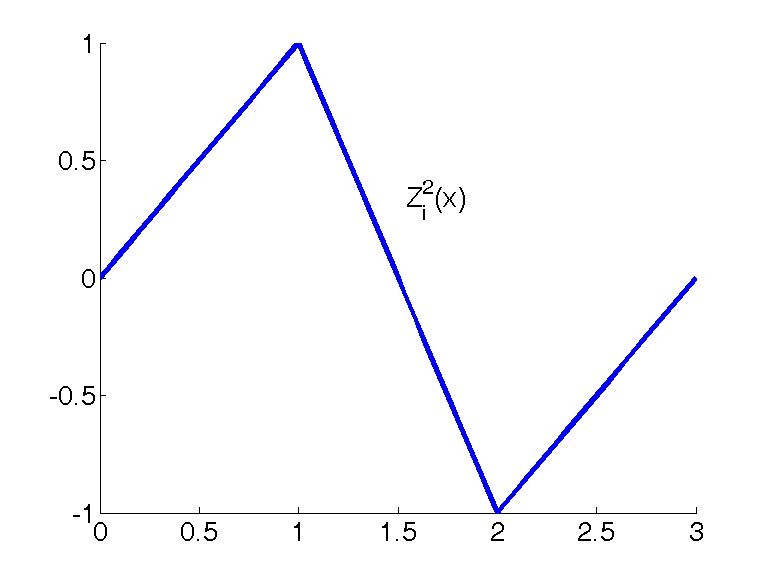}
    \includegraphics[width=6cm,height=4cm]{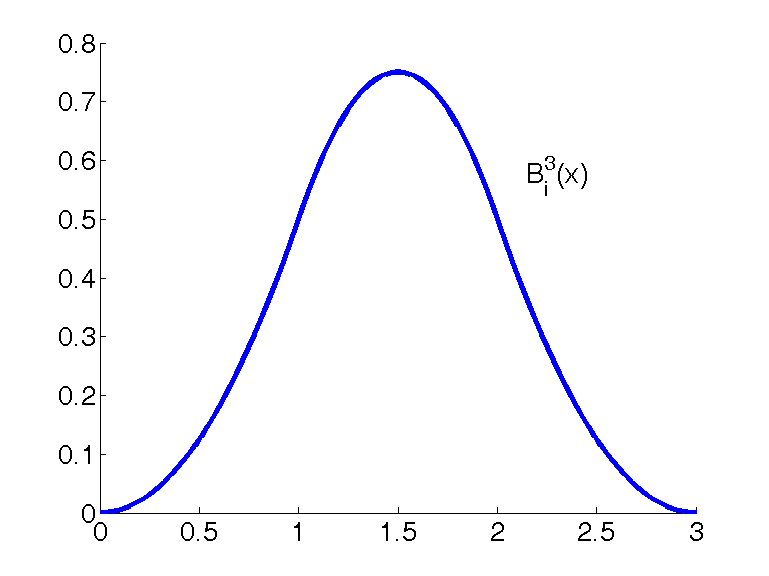}
    \caption{\label{o1} The piecewise linear function $Z_i^2(x)=\dfrac{\mbox{d}}{\mbox{d}x} B_{i}^{3}(x)$ with equidistant knots $0,1,2,3$.}
  \end{center}
\end{figure}

\begin{figure}[!ht]
  \begin{center}
    \includegraphics[width=6cm,height=4cm]{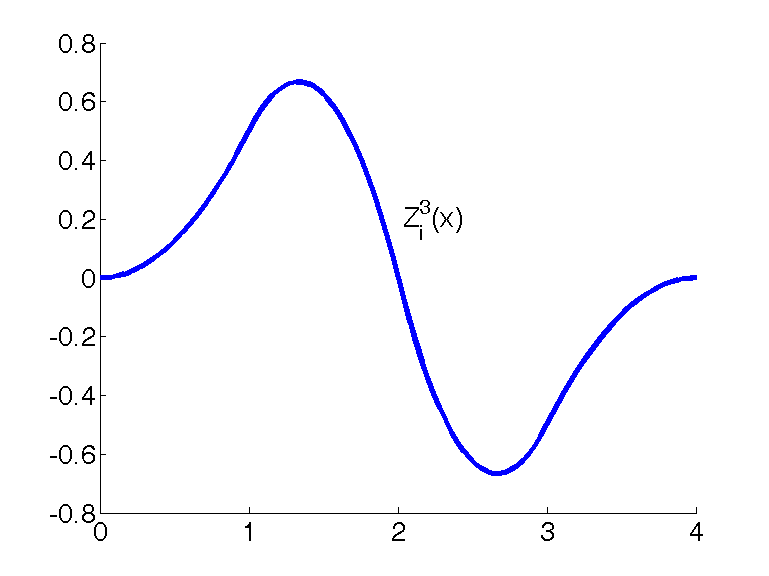}
    \includegraphics[width=6cm,height=4cm]{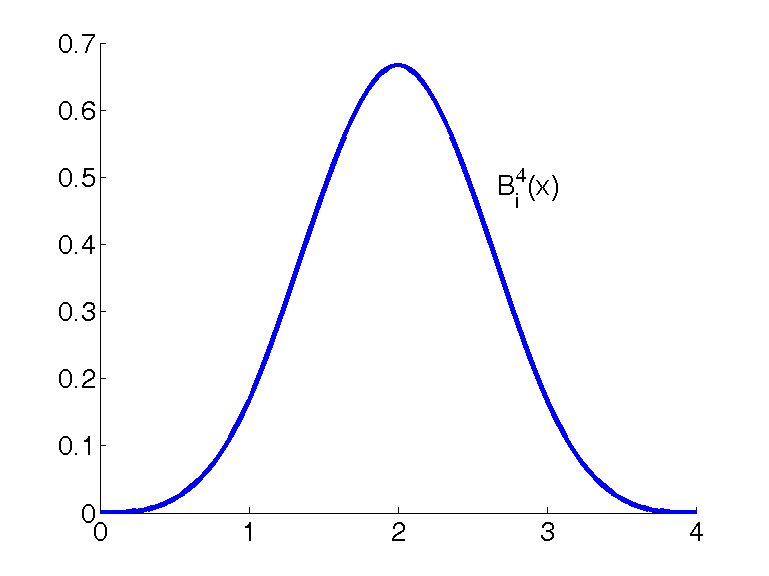}
    \caption{\label{o2} The piecewise quadratic function $Z_i^3(x)=\dfrac{\mbox{d}}{\mbox{d}x} B_{i}^{4}(x)$ with equidistant knots $0,1,2,3,4$.}
  \end{center}
\end{figure}

\begin{figure}[!ht]
  \begin{center}
     \includegraphics[width=6cm,height=4cm]{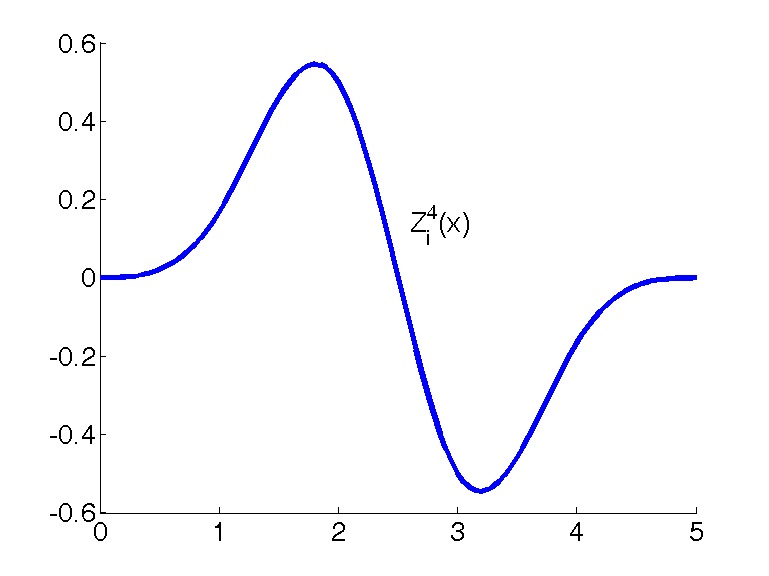}
     \includegraphics[width=6cm,height=4cm]{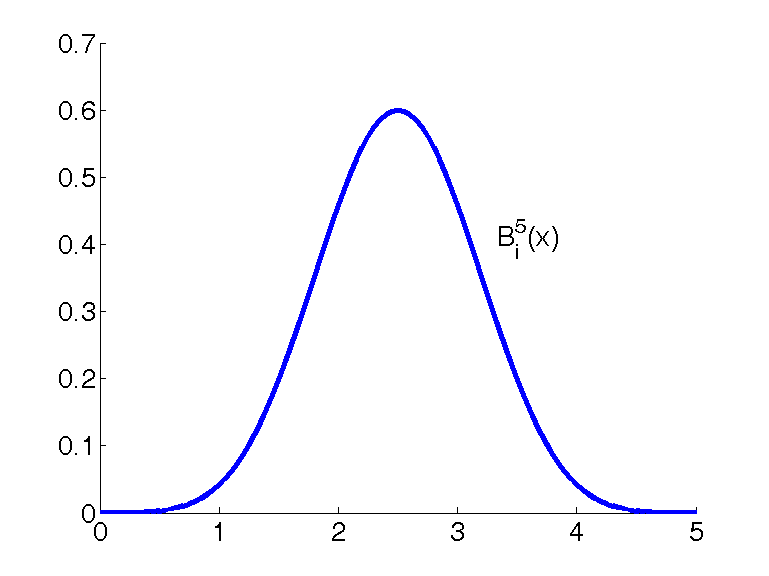}
     \caption{\label{o3} The piecewise cubic function $Z_i^4(x)=\dfrac{\mbox{d}}{\mbox{d}x} B_{i}^{5}(x)$ with equidistant knots $0,1,2,3,4,5$.}
  \end{center}
\end{figure}

\begin{figure}[!ht]
  \begin{center}
    \includegraphics[width=6cm,height=4cm]{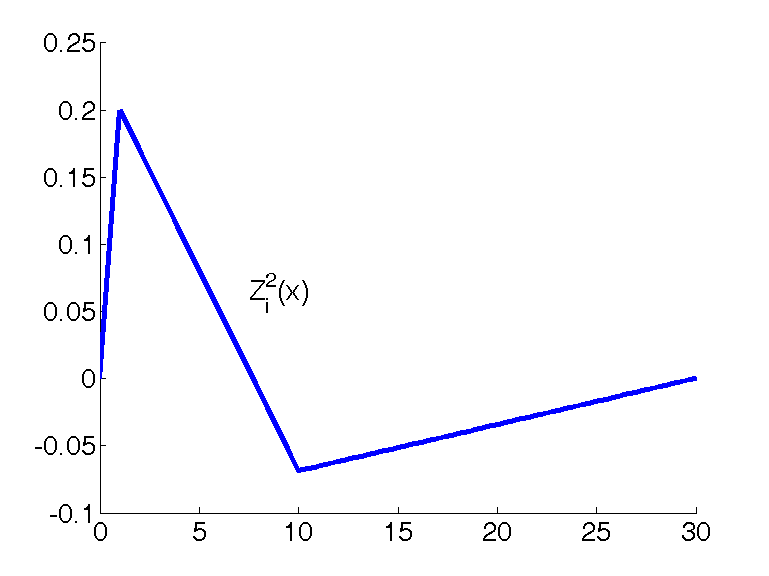}
    \includegraphics[width=6cm,height=4cm]{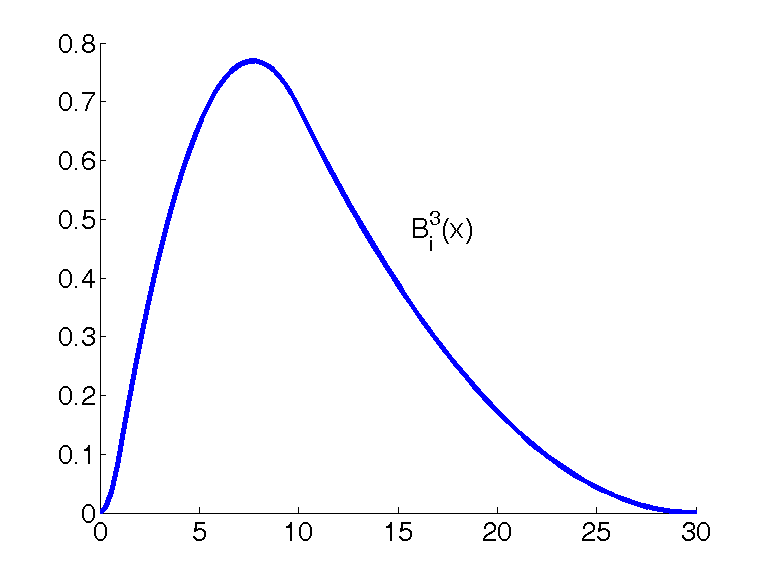}
    \caption{\label{o4} The piecewise linear function $Z_i^2(x)=\dfrac{\mbox{d}}{\mbox{d}x} B_{i}^{3}(x)$ with nonequidistant knots $0,1,10,30$.}
  \end{center}
\end{figure}

\begin{figure}[!ht]
  \begin{center}
     \includegraphics[width=6cm,height=4cm]{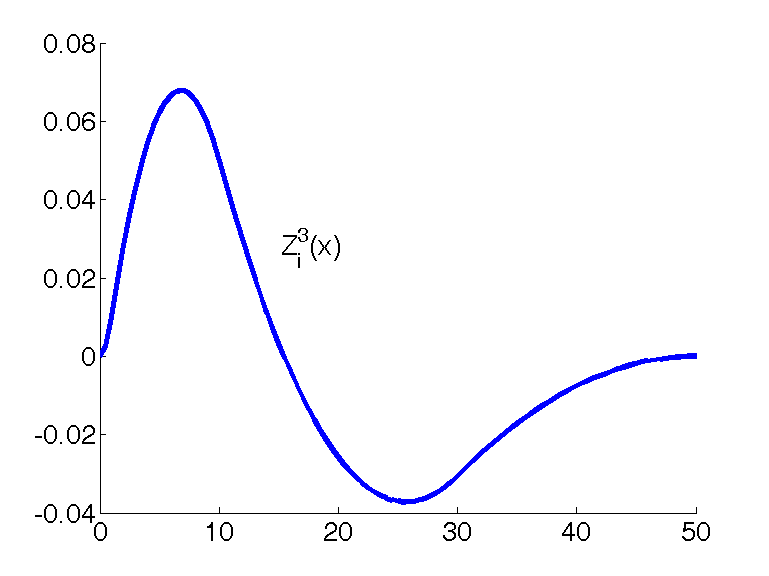}
     \includegraphics[width=6cm,height=4cm]{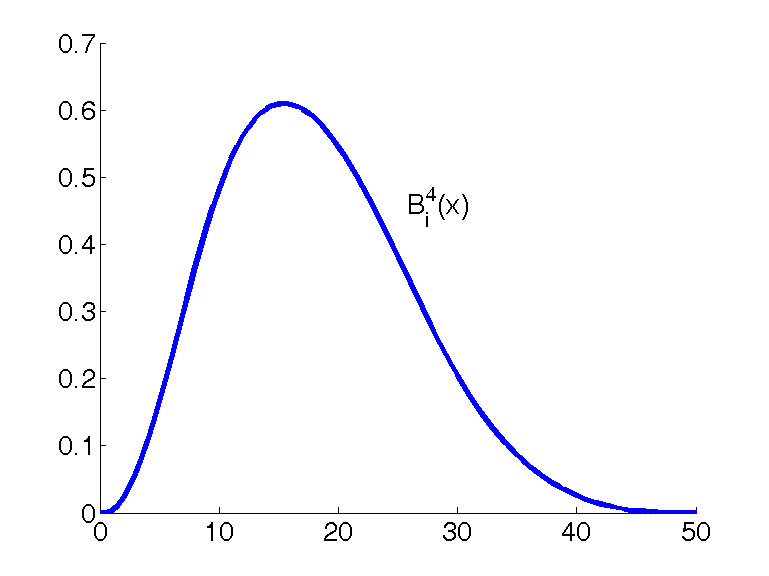}
     \caption{\label{o5} The piecewise quadratic function $Z_i^3(x)=\dfrac{\mbox{d}}{\mbox{d}x} B_{i}^{4}(x)$ with nonequidistant knots $0,1,10,30,50$.}
  \end{center}
\end{figure}

It is known that for the vector space ${\cal S}_{k}^{\Delta\lambda}[a,b]$ of polynomial splines of degree $k>0$, $k\in\mathbb{N}$,
defined on a finite interval $I=[a,b]$ with the sequence of knots $\Delta\lambda= \left\{\lambda_i\right\}_{i=0}^{g+1}$, $\lambda_{0}=a<\lambda_{1}<\ldots<\lambda_{g}<b=\lambda_{g+1}$,
the dimension is $$\dim({\cal S}_{k}^{\Delta\lambda}[a,b]) \;= \; g+k+1.$$
For the construction of all basis functions $B_i^{k+1}(x)$, it is necessary to consider some additional knots.
Without loss of generality we can add coincident knots
\begin{equation}\label{adknot}
\lambda_{-k}=\cdots=\lambda_{-1}=\lambda_{0}=a, \qquad b=\lambda_{g+1}=\lambda_{g+2}=\cdots=\lambda_{g+k+1}.
\end{equation}
Then every spline $s_{k}(x)\in{\cal S}_{k}^{\Delta\lambda}[a,b]$ in $L^2(I)$ has a unique representation
\begin{equation}\label{br}
    s_{k}\left(x\right)=\sum\limits_{i=-k}^{g}b_{i}B_{i}^{k+1}\left(x\right).
\end{equation}
In \cite{Mach, Tal}, the splines with zero integral are studied. There is given the necessary and sufficient condition for $B$-splines
coefficients of these splines. However, typical $B$-splines $B_i^{k+1}(x)$, thus ignoring the constraint (\ref{clrzero}) in $L_0^2(I)$ for construction of the $B$-spline basis, were used there.

\medskip
Now, regarding the definition (\ref{zb}), we are able to use spline functions $Z_i^{k+1}(x)$ which have zero integral 
on $I$ (denoted also as $ZB$-splines in the sequel).
In the following, ${\cal Z}_{k}^{\Delta\lambda}[a,b]$ denotes the vector space of polynomial splines of degree $k>0$, defined on a finite interval
$[a,b]$ with the sequence of knots $\Delta\lambda$ and having zero integral on $[a,b]$, it means
\begin{equation}\label{Zspace}
{\cal Z}_{k}^{\Delta\lambda}[a,b] \, := \, \{ s_k(x) \in {\cal S}_{k}^{\Delta\lambda}[a,b]\, : \, \int\limits_I s_k(x) \, \mbox{d}x = 0 \}.
\end{equation}

\begin{theorem}
  The dimension of the vector space ${\cal Z}_{k}^{\Delta\lambda}[a,b]$ defined by the formula (\ref{Zspace}) is 
  $g+k$.
\end{theorem}

\begin{proof}
  For spline $s_k(x)\in{\cal S}_{k}^{\Delta\lambda}[a,b]$, $s_{k}\left(x\right)=\sum\limits_{i=-k}^{g}b_{i}B_{i}^{k+1}\left(x\right)$,
  with the coincident additional knots it is known, \cite{dierckx93}, that
  $$
  \int\limits_I s_k(x) \, \mbox{d}x \, = \, \dfrac{1}{k+1} \, \sum\limits_{i=-k}^{g} \, b_{i}(\lambda_{i+k+1}-\lambda_i).
  $$
  It means that $B$-spline coefficients of $s_k(x)\in{\cal Z}_{k}^{\Delta\lambda}[a,b] \subset {\cal S}_{k}^{\Delta\lambda}[a,b]$ satisfy condition
  $0=\sum\limits_{i=-k}^{g} \, b_{i}(\lambda_{i+k+1}-\lambda_i) \, = \, \mathbf{A}\mathbf{b}$ with
  $\mathbf{A}=(\lambda_1-\lambda_{-k},\cdots,\lambda_{g+k+1}-\lambda_{g})$, $\mathbf{b}=(b_{-k},\cdots,b_g)^{\top}$.
  And it is obvious that   $\mbox{codim}({\cal Z}_{k}^{\Delta\lambda}[a,b])=1$, thus
  $$\dim({\cal Z}_{k}^{\Delta\lambda}[a,b]) \, = \, \dim({\cal S}_{k}^{\Delta\lambda}[a,b])- \mbox{codim}({\cal Z}_{k}^{\Delta\lambda}[a,b]) \, = \, g+k.$$
 \end{proof}

\begin{theorem}
  For the coincident additional knots (\ref{adknot}), the 
	functions $Z_{-k}^{k+1}(x),$ $\cdots,Z_{g-1}^{k+1}(x)$ form a basis for the space ${\cal Z}_{k}^{\Delta\lambda}[a,b]$.
\end{theorem}

\begin{proof}
  Since $M_i^{k+1}(x)$ form a basis for the spline space ${\cal S}_{k}^{\Delta\lambda}[a,b]$ and $Z_i^{k+1}(x) \, = \, M_i^{k+1}(x) - M_{i+1}^{k+1}(x)$,
  the functions $Z_i^{k+1}(x)$, $i=-k,\ldots,g-1$, are linearly independent and lie in ${\cal Z}_{k}^{\Delta\lambda}[a,b]$ with $\dim({\cal Z}_{k}^{\Delta\lambda}[a,b])=g+k$.
  Therefore $Z_i^{k+1}(x)$, $i=-k,\ldots,g-1$, form a basis for the ${\cal Z}_{k}^{\Delta\lambda}[a,b]$.
\end{proof}

With regard to this theorem,
every spline $s_{k}(x)\in{\cal Z}_{k}^{\Delta\lambda}[a,b]$ has
a unique representation
\begin{equation}\label{zr}
    s_{k}\left(x\right)=\sum\limits_{i=-k}^{g-1}z_{i}Z_{i}^{k+1}\left(x\right).
\end{equation}


Now we can proceed to matrix notation of $s_{k}(x)\in{\cal Z}_{k}^{\Delta\lambda}[a,b]$.
With respect to (\ref{defZ}) and (\ref{ZM}), we are able to write the functions $Z_{i}^{k+1}\left(x\right)$ in matrix notation as
$$
Z_{i}^{k+1}\left(x\right)=(k+1)\left(B_{i}^{k+1}\left(x\right),B_{i+1}^{k+1}\left(x\right)\right)
\left(\begin{array}{cc}
        \dfrac{1}{\lambda_{i+k+1}-\lambda_{i}} & 0 \\
        0 & \dfrac{1}{\lambda_{i+k+2}-\lambda_{i+1}}
      \end{array}
\right)
\left(\begin{array}{r}
        1 \\
        -1
      \end{array}
\right).
$$
Then it is clear that
$$
(Z_{-k}^{k+1}\left(x\right),\ldots,Z_{g-1}^{k+1}\left(x\right))=(B_{-k}^{k+1}\left(x\right),\ldots,B_{g}^{k+1}\left(x\right))\mathbf{D}\mathbf{K}=
\mathbf{B}_{k+1}(x)\mathbf{D}\mathbf{K},
$$
where
\begin{equation}\label{mD}
  \mathbf{D}=(k+1)\mbox{diag}\left(\dfrac{1}{\lambda_{1}-\lambda_{-k}},\ldots,\dfrac{1}{\lambda_{g+k+1}-\lambda_{g}}\right)=
(k+1)\mbox{diag}\left(\dfrac{1}{l_1},\ldots,\dfrac{1}{l_{g+k+1}}\right)
\end{equation}
and
\begin{equation}\label{mK}
  \mathbf{K}=\left(\begin{array}{rrrrrr}
	               1 &  0 & 0 & \cdots & 0 & 0 \\
                  -1 &  1 & 0 & \cdots & 0 & 0 \\
                   0 & -1 & 1 & \cdots & 0 & 0 \\
                   \vdots &  \vdots & \ddots & \ddots & \vdots & \vdots \\
                   0 &  0 & 0 &\cdots & -1 & 1 \\
                   0 &  0 & 0 & \cdots & 0 & -1
                 \end{array}\right)\in\mathbb{R}^{g+k+1,g+k}.
\end{equation}
Therefore the spline $s_{k}(x)\in{\cal Z}_{k}^{\Delta\lambda}[a,b]$, $s_{k}\left(x\right)=\sum\limits_{i=-k}^{g-1}b_{i}Z_{i}^{k+1}\left(x\right)$
can be written in matrix notation as
\begin{equation}\label{mn}
  s_{k}(x) \; = \; \mathbf{Z}_{k+1}(x)\mathbf{z} \; = \; \mathbf{B}_{k+1}(x)\mathbf{D}\mathbf{K}\mathbf{z},
\end{equation}
where
$\mathbf{Z}_{k+1}(x)=(Z_{-k}^{k+1}\left(x\right),\ldots,Z_{g-1}^{k+1}\left(x\right))$
and
$\mathbf{z}=\left(z_{-k},\ldots,z_{g-1}\right)^{\top}$.

\begin{remark}
The formula (\ref{mn}) is very useful, because we can use the standard $B$-spline 
basis for working with splines honoring the zero integral constraint, which is very convenient from a computational point of view.
\end{remark}


\begin{example}
We consider knots $\Delta\lambda=\left\{\lambda_i\right\}_{i=0}^{g+1}$, $\lambda_{0}=0=a<2<5<9<14<b=20=\lambda_{5}$.
The task is to find a cubic spline with the given sequence of knots and which has zero integral on the interval $[0,20]$.
It is evident that $k=3$, $g=4$. We consider the additional knots
$$\lambda_{-3}=\lambda_{-2}=\lambda_{-1}=\lambda_{0}=a=0, \qquad 20=b=\lambda_{5}=\lambda_{6}=\lambda_{7}=\lambda_{8}.$$
The basis functions of the  space ${\cal Z}_{3}^{\Delta\lambda}[0,20]$ are plotted in Figure \ref{o6}.
Every spline $s_{3}(x)\in{\cal Z}_{3}^{\Delta\lambda}[0,20]$ can be written as
\begin{equation}\label{zc}
    s_{3}\left(x\right)=\sum\limits_{i=-3}^{3}z_{i}Z_{i}^{4}\left(x\right).
\end{equation}
Thus, e.g., for
$\mathbf{z}=(z_{-3},\ldots,z_3)^{\top}=(0.5,-1,2,3,-8,9,1)^{\top}$
the cubic spline $s_3(x)$ with zero integral is plotted in Figure \ref{o7}.
\end{example}

\begin{figure}[!ht]
  \begin{center}
    \includegraphics[width=8.5cm,height=6cm]{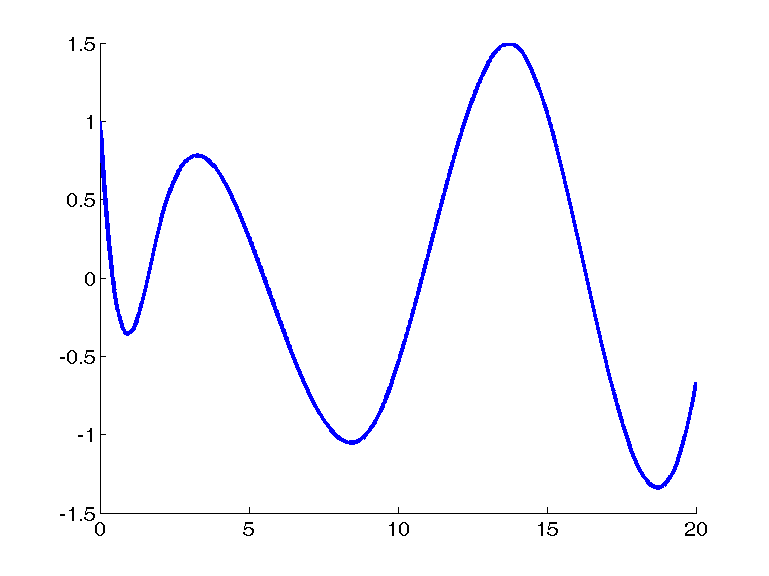}
     \caption{\label{o7} Cubic spline $s_3(x)$ with given coefficients $\mathbf{z}=(0.5,-1,2,3,-8,9,1)^{\top}$.}
  \end{center}
\end{figure}

\begin{figure}[!ht]
  \begin{center}
    \includegraphics[width=8.5cm,height=6cm]{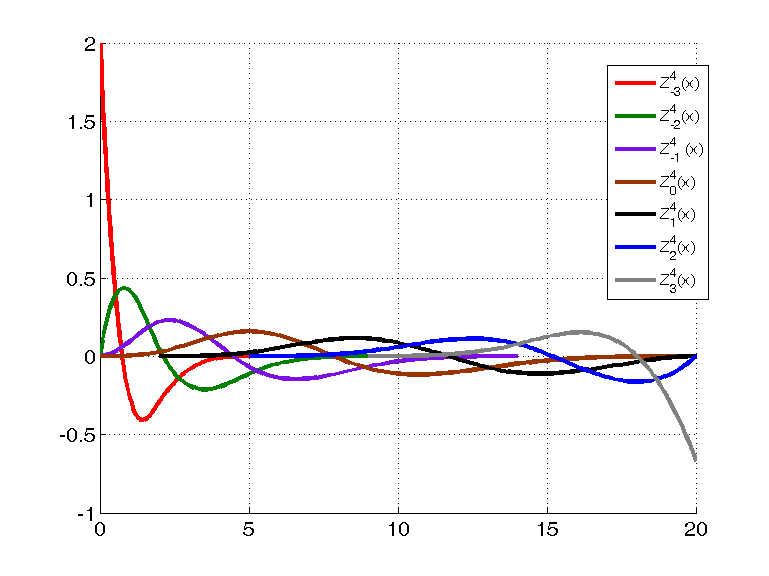}
    \caption{\label{o6} Basis splines for the space ${\cal Z}_{3}^{\Delta\lambda}[0,20]$.}
  \end{center}
\end{figure}


\section{Smoothing spline in $L_0^2(I)$}
\label{smoothing}


In \cite{Mach}, the construction of smoothing splines in the space $L_0^2(I)$ was studied, however using standard $B$-spline basis functions $B_i^{k+1}(x)$. Now we are able to construct smoothing splines in this space with new basis functions $Z_i^{k+1}(x)$. For this purpose, let
data $(x_{i},y_{i})$, $a\leq x_{i}\leq b$, weights $w_{i}>0$, $i=1,\ldots,n$,
sequence of knots $\Delta\lambda=\left\{\lambda_i\right\}_{i=0}^{g+1}$, $\lambda_{0}=a<\lambda_{1}<\ldots<\lambda_{g}<b=\lambda_{g+1}$, $n\geq g+1$ 
and a parameter $\alpha\in(0,1)$ be given. For arbitrary $l\in\left\{1,\ldots,k-1\right\}$ our task is to find a spline 
$s_{k}(x)\in{\cal Z}_{k}^{\Delta\lambda}[a,b]\subset L_0^2(I)$, which minimizes the functional
\begin{equation*}
    J_{l}(s_k) \, = \, (1-\alpha)\int_{a}^{b}\left[s_{k}^{(l)}(x)\right]^{2}\, \mbox{d}x +  \alpha
    \sum\limits_{i=1}^{n} w_{i}\left[y_{i}-s_{k}(x_{i})\right]^{2}.
\end{equation*}
Note that the choice of parameter $\alpha$ and $l$, where $l$ stands for $l$th derivation, affects smoothness of the resulting spline.
Let us denote $\mathbf{x}=\left(x_{1},\ldots,x_{n}\right)^{\top}$, $\mathbf{y}=\left(y_{1},\ldots,y_{n}\right)^{\top}$,
$\mathbf{w}=\left(w_{1},\ldots,w_{n}\right)^{\top}$ and $\mathbf{W}=diag\left(\mathbf{w}\right)$.
Regarding the representation (\ref{zr}) and matrix notation (\ref{mn}), the functional $J_l(s_k)$ can be written as a quadratic function
\begin{equation}\label{jl}
\begin{split}
  J_l(\mathbf{z}) \ = & \, (1-\alpha){\bf z}^{\top}{\bf K}^{\top}{\bf D}{\bf S}_l^{\top}{\bf M}_{kl}{\bf S}_l {\bf DKz}\, + \\
	& + \alpha \left[{\bf y}-{\bf B}_{k+1}({\bf x}){\bf DKz}\right]^{\top} {\bf W} \left[{\bf y}-{\bf B}_{k+1}({\bf x}){\bf DKz}\right],
\end{split}
\end{equation}
see \cite{Mach2,Mach1, Mach} for details. In fact the matrix
\begin{equation}\label{mkl}
  {\bf M}_{kl}=\left( m_{ij}^{kl}\right)_{i,j=-k+l}^{g}, \qquad\mbox{with}\qquad m_{ij}^{kl}=\int\limits_{a}^{b}B_{i}^{k+1-l}(x)B_{j}^{k+1-l}(x)\,\mbox{d}x
\end{equation}
is positive definite, because $B_{i}^{k+1-l}(x)\geq 0$, $i=-k+l,\ldots,g$ are basis functions.
Upper triangular matrix ${\bf S}_{l}={\bf D}_{l}{\bf L}_{l}\ldots{\bf D}_{1}{\bf L}_{1}\in\mathbb{R}^{g+k+1-l,g+k+1}$ has full row rank.
${\bf D}_{j}\in\mathbb{R}^{g+k+1-j,g+k+1-j}$ is a diagonal matrix such that
$$
{\bf D}_{j}=\left(k+1-j\right)diag\left(d_{-k+j},\ldots,d_{g}\right)
$$
with
$$
d_{i}=\dfrac{1}{\lambda_{i+k+1-j}-\lambda_{i}}, \quad i=-k+j,\ldots,g,
$$
and
$${\bf L}_{j}:=\left(\begin{array}{cccc}
                           -1 & 1      &        &  \\
                              & \ddots & \ddots &  \\
                              &        & -1     & 1
                         \end{array} \right)\in\mathbb{R}^{g+k+1-j,g+k+2-j}.
$$
Finally, ${\bf B}_{k+1}({\bf x})\in\mathbb{R}^{n,g+k+1}$ stands for the collocation matrix, i.e.
$$
{\bf B}_{k+1}({\bf x}) \, = \, \left(B_i^{k+1}(x_j)\right)_{j=1,\, i=-k}^{n,\, g}.
$$
Using the notation $\mathbf{U}:= {\bf D}{\bf K}$,
\begin{equation}\label{vg}
  {\bf G} \, := \, {\bf U}^{\top}\left[(1-\alpha){{\bf S}_l}^{\top}{\bf M}_{kl}{\bf S}_l+\alpha{\bf B}_{k+1}^{\top}({\bf x}){\bf W}{\bf B}_{k+1}({\bf x})\right]{\bf U}
\end{equation}
and
$$
{\bf g} \, := \, \alpha {\bf K}^{\top}{\bf D}{\bf B}_{k+1}^{\top}({\bf x}){\bf Wy},
$$
it is possible to rewrite the quadratic function $J_l(\bf z)$ as
\begin{equation}\label{jlz}
  J_l(\bf z)\, = \, {\bf z}^{\top}{\bf Gz}-2{\bf z}^{\top}{\bf g}+\alpha{\bf y}^{\top}{\bf Wy}.
\end{equation}
Our task is to find a spline $s_{k}(x)\in{\cal Z}_{k}^{\Delta\lambda}[a,b]$ which minimizes the functional $J_{l}(s_{k})$, in other words, we want to
find a minimum of the function (\ref{jlz}). It is obvious that this function has just one minimum if and only if the matrix
${\bf G}$ is positive definite (p.d.). From (\ref{vg}) it can be easily seen
that
$$
{\bf G} \quad \mbox{ is p.d.} \quad \Leftrightarrow \quad {\bf B}_{k+1}(\bf x) \quad \mbox{is of full column rank}.
$$
From Schoenberg-Whitney theorem and its generalization, see \cite{deboor78} and \cite{Mach2}, it is known that
matrix ${\bf B}_{k+1}(\bf x)$ is of full column rank if and only if there exists $\{u_{-k},\ldots, u_{g}\}\subset\{x_1,\ldots,x_n\}$ with
$u_i<u_{i+1}$, $i=-k,\ldots,g-1$, such that $\lambda_i<u_i<\lambda_{i+k+1}$, $i=-k,\ldots,g$.
In this case from the necessary and sufficient condition for a unique minimum of quadratic function, i.e.
$$
\frac{\partial J_{l}({\bf z})}{\partial{\bf z}^{\top}}=0,
$$
we get a system of linear equations ${\bf Gz}={\bf g}$ and then the unique solution of this system is given by
\begin{equation}\label{z*}
    {\bf z}^{*}={\bf G}^{-1}{\bf g}.
\end{equation}
Consequently, the resulting smoothing spline is obtained by the formula
\begin{equation*}\label{sz*}
    s_{k}^{*}(x) \, = \, \sum\limits_{i=-k}^{g-1}z_{i}^{*}Z_{i}^{k+1}(x),
\end{equation*}
in matrix notation using standard $B$-splines $B_i^{k+1}(x)$ as
\begin{equation*}\label{szm*}
  s_{k}^{*}(x) \, = \, \mathbf{B}_{k+1}(t)\mathbf{DK}\mathbf{z}^*,
\end{equation*}
where the vector $\mathbf{z}^{*}=\left(z_{-k}^{*},\ldots,z_{g-1}^{*}\right)^{'}$ is obtained as

.



\section{Orthogonalization of basis functions}
\label{orthogonalization}


A further step is to orthogonalize the basis
$$\mathbf{Z}_{k+1}(x)=(Z_{-k}^{k+1}\left(x\right),\ldots,Z_{g-1}^{k+1}\left(x\right))^{\top}$$
of the space
${\cal Z}_{k}^{\Delta\lambda}[a,b]$ that is by construction obligue with respect to the $L^2$ space metric. For this purpose the idea presented in \cite{redd12} is used.
We search for a linear transformation $\boldsymbol\Phi$ such that
$$
\mathbf{O}_{k+1}(x)=\boldsymbol\Phi\mathbf{Z}_{k+1}(x)
$$
forms an orthogonal set of basis functions of the space ${\cal Z}_{k}^{\Delta\lambda}[a,b]$, i.e.
$$
\int\limits_{a}^{b}\mathbf{O}_{k+1}(x)\mathbf{O}^{\top}_{k+1}(x)\,\mbox{d}x \; = \; \mathbf{I}.
$$

Regarding the lemma presented in \cite{redd12} and notation used here, we can formulate the following statement.

\begin{lemma}
An invertible transformation $\boldsymbol\Phi$ orthogonalizes the basis functions $\mathbf{Z}_{k+1}(x)$ if and only if it satisfies the condition
that
$$
\boldsymbol\Phi^{\top}\boldsymbol\Phi \; = \; \boldsymbol\Sigma^{-1},
$$
where $\boldsymbol\Sigma$ represents the positive definite matrix
$$
\boldsymbol\Sigma \; = \;
\int\limits_{a}^{b}\mathbf{Z}_{k+1}(x)\mathbf{Z}^{\top}_{k+1}(x)\,\mathrm{d}x \; =
\left( \int\limits_{a}^{b}Z^{k+1}_{i}(x)Z^{k+1}_{j}(x)\,\mathrm{d}x \right)_{i,j=-k}^{g-1}.
$$
\end{lemma}

With respect to the definition of basis functions $\mathbf{Z}_{k+1}(x)=\left(Z_{-k}^{k+1}\left(x\right),\ldots,\right.$ $\left.
Z_{g-1}^{k+1}\left(x\right)\right)^{\top}$
the matrix $\boldsymbol\Sigma$ can be expressed as
\begin{equation}\label{sigma}
  \boldsymbol\Sigma \; = \; \mathbf{K}^{\top} {\bf D}\mathbf{M}\mathbf{DK},
\end{equation}
where ${\bf M}\, :=\, {\bf M}_{k0}$. The linear transformation $\boldsymbol\Phi$ is not unique and can be computed for example by the Cholesky
decomposition. The basis functions
$$
\mathbf{O}_{k+1}(x)=\boldsymbol\Phi\mathbf{Z}_{k+1}(x),\qquad \mathbf{O}_{k+1}(x)=(O_{-k}^{k+1}\left(x\right),\ldots,O_{g-1}^{k+1}\left(x\right))^{\top}
$$
are orthogonal and have a zero integral. The linear and quadratic $ZB$-splines with zero integral and their orthogonalization are plotted in Figures
\ref{o8} and \ref{o9}.

\begin{figure}[!ht]
  \begin{center}
    \includegraphics[width=15cm,height=6cm]{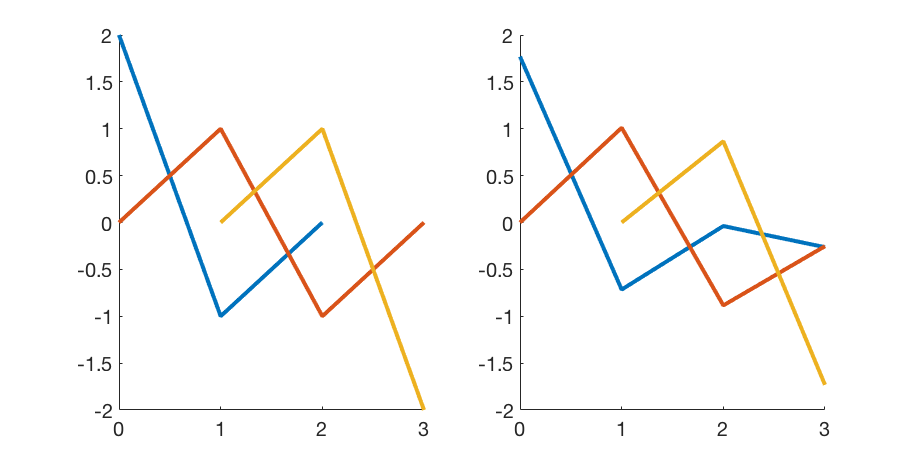}
     \caption{\label{o8} Linear $ZB$-splines $Z_i^{2}(x)$ with given knots $0,0,1,2,3,3$ (left), linear orthogonal $ZB$-splines $O_i^{2}(x)$ (right).}
  \end{center}
\end{figure}

\begin{figure}[!ht]
  \begin{center}
    \includegraphics[width=15cm,height=6cm]{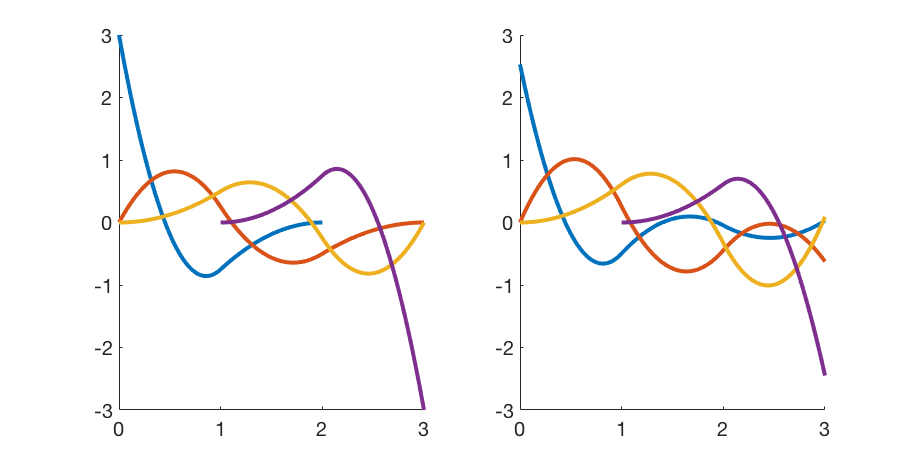}
     \caption{\label{o9} Quadratic $ZB$-splines $Z_i^{3}(x)$ with given knots $0,0,0,1,2,3,4,4,4$ (left), quadratic orthogonal $ZB$-splines $O_i^{3}(x)$ (right).}
  \end{center}
\end{figure}

To sum up, the spline $s_{k}(x)$ with zero integral can be constructed as a linear combination of orthogonal splines $O_i^{k+1}(x)$ having zero integral in a form
$$
s_{k}(x) \; = \; \sum\limits_{i=-k}^{g-1}z_i O_{i}^{k+1}\left(x\right) \; = \; \mathbf{O}_{k+1}(x)\mathbf{z}.
$$
On the other hand, the standard and well-known $B$-splines $B_i^{k+1}(x)$ can be used to represent $s_{k}(x)\in{\cal Z}_{k}^{\Delta\lambda}[a,b]$
in matrix form
$$
s_{k}(x) \; = \; \boldsymbol\Phi\mathbf{B}_{k+1}(x)\mathbf{DK}\mathbf{z}.
$$
This formulation seems to be very useful because it allows us to use existing $B$-spline codes in software R or Matlab, for example, the collocation matrix or computation of integrals in (\ref{mkl}).


\section{Compositional splines in the Bayes space}


Construction of splines directly in $L_0^2(I)$ has important practical consequences, however, it is crucial also from the theoretical perspective.
Expressing splines as functions in $L_0^2(I)$ enables to back-transform them to the original Bayes space $\mathcal{B}^2(I)$ using
(\ref{fclrinv}). It results in \textit{compositional $CB$-splines}, obtained from (\ref{defZ}) as
\begin{equation}
\label{sbplines}
\zeta_i^{k+1}(x)=\frac{\exp[Z_i^{k+1}(x)]}{\int_I\exp[Z_i^{k+1}(y)]\,\mbox{d}y},\quad i=-k,\ldots,g-1,\ k\geq 0.
\end{equation}
Accordingly, for instance $ZB$-splines from Figures \ref{o1}-\ref{o5} can be now expressed directly in the Bayes space $\mathcal{B}^2(I)$ as
$CB$-splines, see Figures \ref{o10}-\ref{o12}. Note that $CB$-splines $\zeta_i^{k+1}(x)$ fulfill the unit integral constraint which is, however, not necessary for further considerations. As a consequence, it is
immediate to define vector
space ${\cal C}_{k}^{\Delta\lambda}[a,b]$ of compositional polynomial splines of degree $k>0$, defined on a finite interval $[a,b]$ with the sequence
of knots $\Delta\lambda$. From isomorphism between ${\cal C}_{k}^{\Delta\lambda}[a,b]$ and ${\cal Z}_{k}^{\Delta\lambda}[a,b]$
it holds that
$$
\dim\left({\cal C}_{k}^{\Delta\lambda}[a,b]\right)=g+k.
$$
Moreover, from isometric properties of clr transformation it follows that every compositional spline $\xi_k(x)\in{\cal C}_{k}^{\Delta\lambda}[a,b]$ in
$\mathcal{B}^2(I)$
can be uniquely represented as
\begin{equation}
\label{ssplinerep}
\xi_k(x)=\bigoplus_{i=-k}^{g-1}z_i\odot\zeta_i^{k+1}(x).
\end{equation}
$CB$-splines $\zeta_i^{k+1}(x)$ forming the basis are by the default setting (\ref{defZ}) not orthogonal. Their orthogonalization can be done as
described in Section \ref{orthogonalization}, i.e. by employing $L_0^2(I)$ with the back-transformation to $\mathcal{B}^2(I)$.

\begin{figure}[!ht]
  \begin{center}
    \includegraphics[width=6cm,height=4cm]{linear_Z.png}
    \includegraphics[width=6cm,height=4cm]{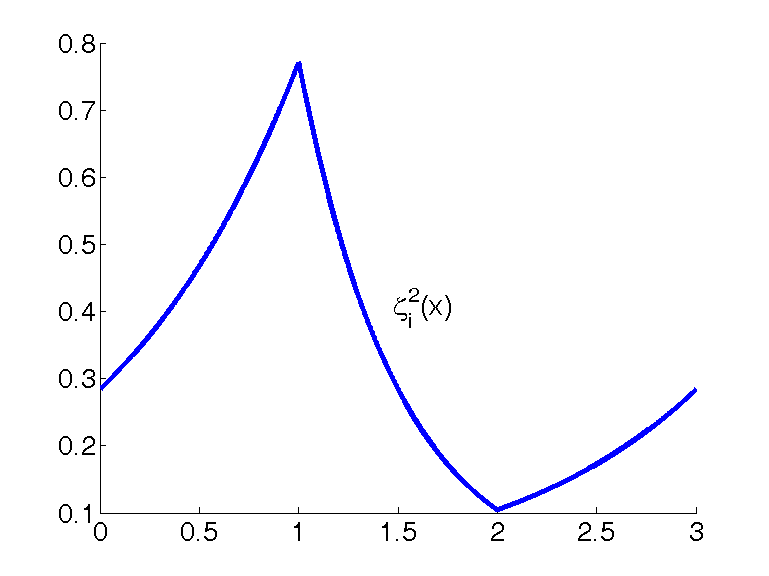}
     \caption{\label{o10} Linear $ZB$-spline $Z_i^2(x)$ (left) and linear $CB$-spline $\zeta_i^2(x)$ (right) with equidistant knots $0,1,2,3$.}
  \end{center}
\end{figure}

\begin{figure}[!ht]
  \begin{center}
    \includegraphics[width=6cm,height=4cm]{quadratic_Z.png}
    \includegraphics[width=6cm,height=4cm]{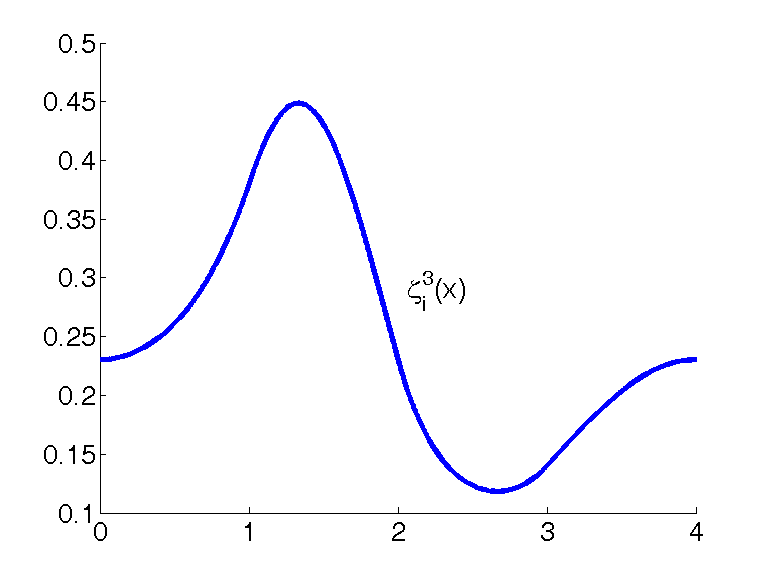}
     \caption{\label{o11} Quadratic $ZB$-spline $Z_i^3(x)$ (left) and quadratic $CB$-spline $\zeta_i^3(x)$ (right) with equidistant knots $0,1,2,3,4$.}
  \end{center}
\end{figure}

\begin{figure}[!ht]
  \begin{center}
    \includegraphics[width=6cm,height=4cm]{quadratic_Z_ne.png}
    \includegraphics[width=6cm,height=4cm]{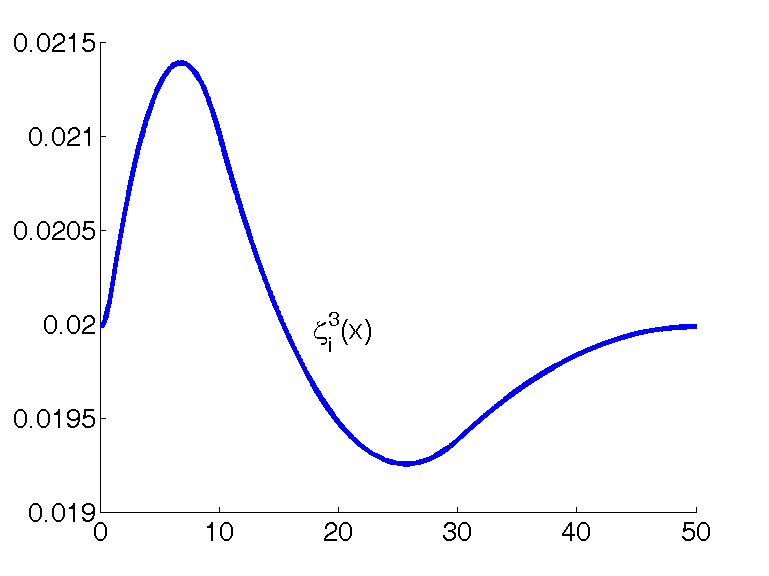}
     \caption{\label{o12} Quadratic $ZB$-spline $Z_i^3(x)$ (left) and quadratic $CB$-spline $\zeta_i^3(x)$ (right) with nonequidistant knots $0,1,10,30,50$.}
  \end{center}
\end{figure}

The resulting compositional splines (with either orthogonal, or non-ortho\-gonal $CB$-spline basis) can be used for representation of densities
directly in $\mathcal{B}^2(I)$. This is an important step in construction of methods of functional data analysis involving density functions, like
for ANOVA modeling \cite{boogaart14} or for the SFPCA method introduced in \cite{hron16} and demonstrated in the next section.
In the latter case $CB$-splines were indeed already used
for construction of the procedure, although at the respective level of development the authors were not aware of that. With $CB$-splines one has a
guarantee that methods are developed consistently in the Bayes space. Moreover, the possibility of having an orthogonal basis enables to gain additional
features resulting from orthogonality of finite dimensional projection in combination with approximate properties of spline functions.

As usual, compositional splines can be tuned according to concrete problem, with the advantage of their direct formulation in the Bayes space sense.

\begin{example}
To illustrate smoothing of concrete data with a compositional spline, 1000 values from standard normal distribution were simulated and the support was determined by minimum and maximum simulated values, $I=[x_{min},x_{max}]$. Data were collected in a form of histogram, where the breakpoints are equidistantly spaced. The representative points of the histogram cells are denoted as asterisks in Figure \ref{o13} (left) and these points form the input data $(x_i,y_i)$, $i=1,\ldots,n$ for smoothing purpose. The $y$-values stand for discretized relative contributions to the overall probability mass, therefore discrete clr transformation \cite{aitchison86} is needed to obtain a real vector with zero sum constraint (Figure \ref{o13}, right). These data points were smoothed using the procedure from Section \ref{smoothing} and back-transformed to the original space. In the concrete setting $k=2$, $l=1$, $\alpha=0.5$,
$\Delta\lambda=\{x_1,-2,-1,0,1,2,x_n\}$, $w_i=1$, $\forall i=1,\ldots,n$ were considered. The resulting spline $s_2(x)$ with zero integral on interval
$[x_{min},x_{max}]$ is plotted in Figure \ref{o13} (right). In the left plot the compositional spline $\xi_2(x)$ with unit integral by using (\ref{fclrinv}) is depicted.

\begin{figure}[!ht]
  \begin{center}
    \includegraphics[width=6cm,height=4cm]{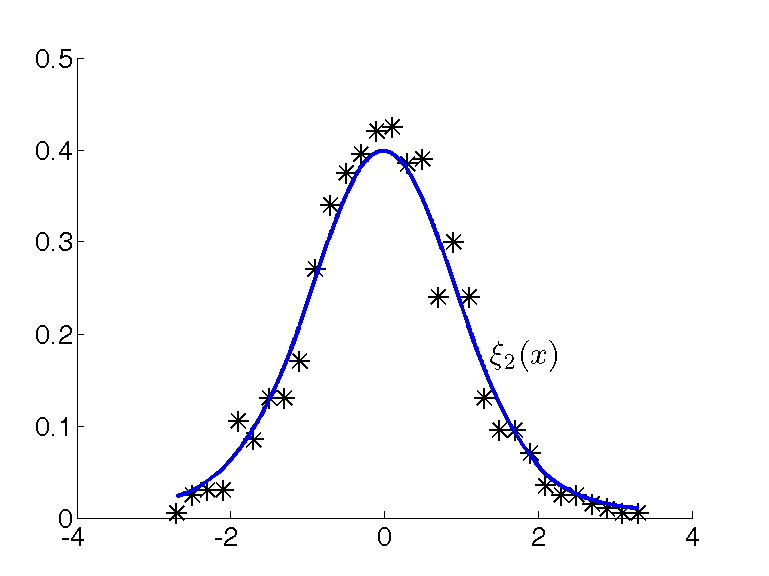}
    \includegraphics[width=6cm,height=4cm]{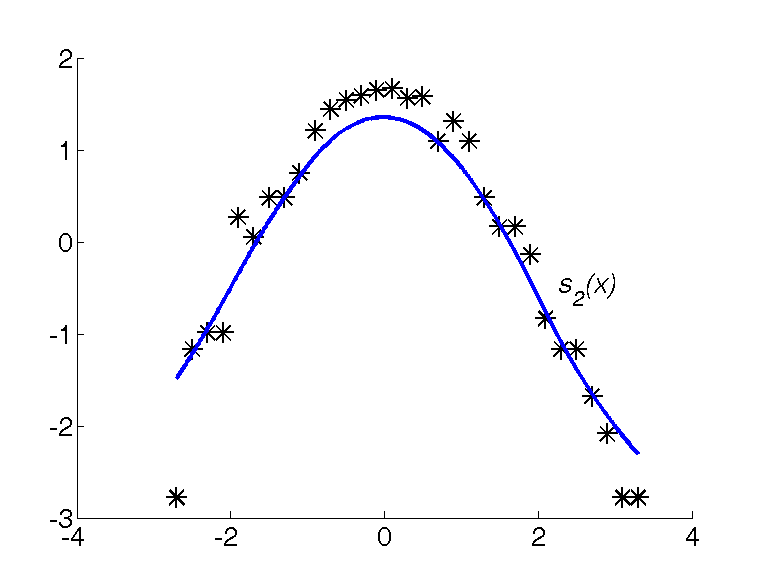}
    \caption{\label{o13} Smoothing of simulated standard normal values in the Bayes space (left) and for clr transformed data (right).}
  \end{center}
\end{figure}
\end{example}


\section{Application to anthropometric data}
For the purpose of illustrating the smoothing procedure outlined in Section \ref{smoothing}, a real-world data set dealing with the most commonly used anthropometric measure relating to body weight is presented. The data set we consider collects the body weight of apparently healthy Czech adolescents and young adults aged 15-31 years (the total of 4436 records) which were recruited non-randomly by offering free body composition assessment. Body weight was measured by the InBody 720 device (Biospace Co., Ltd, Seoul, Korea), recorded as the total body mass rounded to the nearest 0.1 kg.

The raw data for each of $N=16$ age groups, i.e. $\left[15,16\right), \left[16,17\right),\ldots$, $\left[30,31\right)$, were turned into a form of histogram data as follows. The sampled values of the body weight in each age group were divided into equally-spaced classes of the united range 40-110 kg and the optimal number of classes, denoted by $q_i$, $i=1,\ldots,16$, was set according to the well-known Sturge's rule separately across the age groups. Because of the insufficient number of sampled data for males and females in each age group, gender information was ignored. Although there might be some differences between male and female samples due to sex dimorphism, they are not that dramatic (e.g., contrary to body height) because the weight is influenced also by external factors (nutrition, physical activity), and still allow for a reasonable aggregation of data. Subsequently, the proportions in classes within each age group were computed and present zero-values (zero counts in the respective classes) were imputed by values $(2/3)\cdot(1/n_i),i=1, \ldots, N$ according to \cite{martin15}, where $n_i$ stands for the number of observations in $i$th age group. Finally, the raw discretized density data $f_{i,j}, i=1,\ldots,N, j=1, \ldots, q_i$ which correspond to the midpoints $t_{i,j}$ of classes, $i=1,\ldots,N, j=1, \ldots, q_i$ (i.e., $f_{i,j} = f(t_{i,j}))$, were obtained by dividing (not necessary normalized) proportions $\textbf{p}_{i}=\left(p_{i,1},\ldots, p_{i,q_i}\right)^{\top}$ of counts in classes by the length of the respective intervals resulting from partition of the weight range in each of age groups. Figure \ref{fig:data_hist} shows four examples of histograms with different number of classes together with raw data (Table \ref{tab:rawdata}) to be smoothed. To do so, their transformation into real vectors is needed. We note that if the histogram data are constructed on subintervals of the same length, i.e. with equally-spaced breakpoints, it enables to use the discrete version of the clr transformation \cite{aitchison86} directly on the vector of proportions $\textbf{p}_{i}, i=1,\ldots,N$ by considering the scale invariance property. Otherwise, the input of the clr transformation must be vectors with raw density data $\textbf{f}_i=\left(f_{i,1},\ldots, f_{i,q_i}\right)^{\top}, i=1,\ldots,N$ in order to avoid misleading results which would not reflect the actual behavior of data. Vectors of clr transforms are hereafter denoted as $\clr (\textbf{f}_i)=\left(\clr (f_{i,1}),\ldots, \clr (f_{i,q_i})\right)^{\top}$ for $i=1,\ldots,N$ and clr values are listed in Table \ref{tab:clrdata}.

Having the collected data $(t_{i,j},\clr (f_{i,j}))$, we proceed to smooth them with the compositional smoothing splines using a system of $ZB$-spline basis functions from the $L^2_0(I)$. They are considered on domain $I = \left[40,107\right]$ which has been modified in order to avoid undesired artifacts in densities at their right-hand side. For all $N$ observations, the same strategy was followed to set the values of the input parameters for the smoothing procedure. We employed cubic smoothing splines ($k=3,l=2$) with the given sequence of knots $\Delta \lambda =\{40,62,84,107\}$, the vectors of weights $\textbf{w}_{i}$ for all input data equal to vectors of ones and the smoothing parameter $\alpha$ was set to $0.5$. That is, when minimizing the penalized functional (\ref{jl}), the same importance is assigned to both smoothness of the smoothing splines as well as to their approximative properties. The resulting compositional smoothing splines are obtained via their clr representation
\begin{equation} \label{SBspline-weight}
s_3^i(t) = \sum_{\nu=-3}^{1} z_{i,\nu} Z_{\nu}^{4}(t), \quad i= 1,\ldots, N, \quad t \in I;
\end{equation}
the corresponding $ZB$-spline coefficients are reported in Table \ref{tab:ZB-coef}.

Figure \ref{fig:data_spliny_example} displays an example of three raw density data from Figure \ref{fig:data_hist} together with smoothed curves in the $L^2_0$ space (right) and after the inverse transformation (\ref{fclrinv}) in the $\mathcal{B}^2$ space (left). The whole sample of smoothed density functions (Figure \ref{fig:data_spliny_all}) is plotted on blue scale distinguishing the age groups. Two trends are apparent -- in the younger age groups, the estimated density functions are right skewed and exhibit lower variability while with increasing age they become more symmetric followed by higher variability. Nevertheless, density function in age group $\left[23,24\right)$ does not fully respect this behavior: the variability trend holds, but the distribution of weights is more similar to those in younger age groups as it is skewed more to the left. In general, it seems that adolescents appear to be predominantly of a lower body weight than the older persons whose weight is more spread over the weight classes and more pronounced in the middle part of the distribution. Accordingly, there is also a higher incidence of higher weights in comparison with the younger adolescents.

\begin{figure}[t!]
    \centering
     \includegraphics[width=0.5\textwidth]{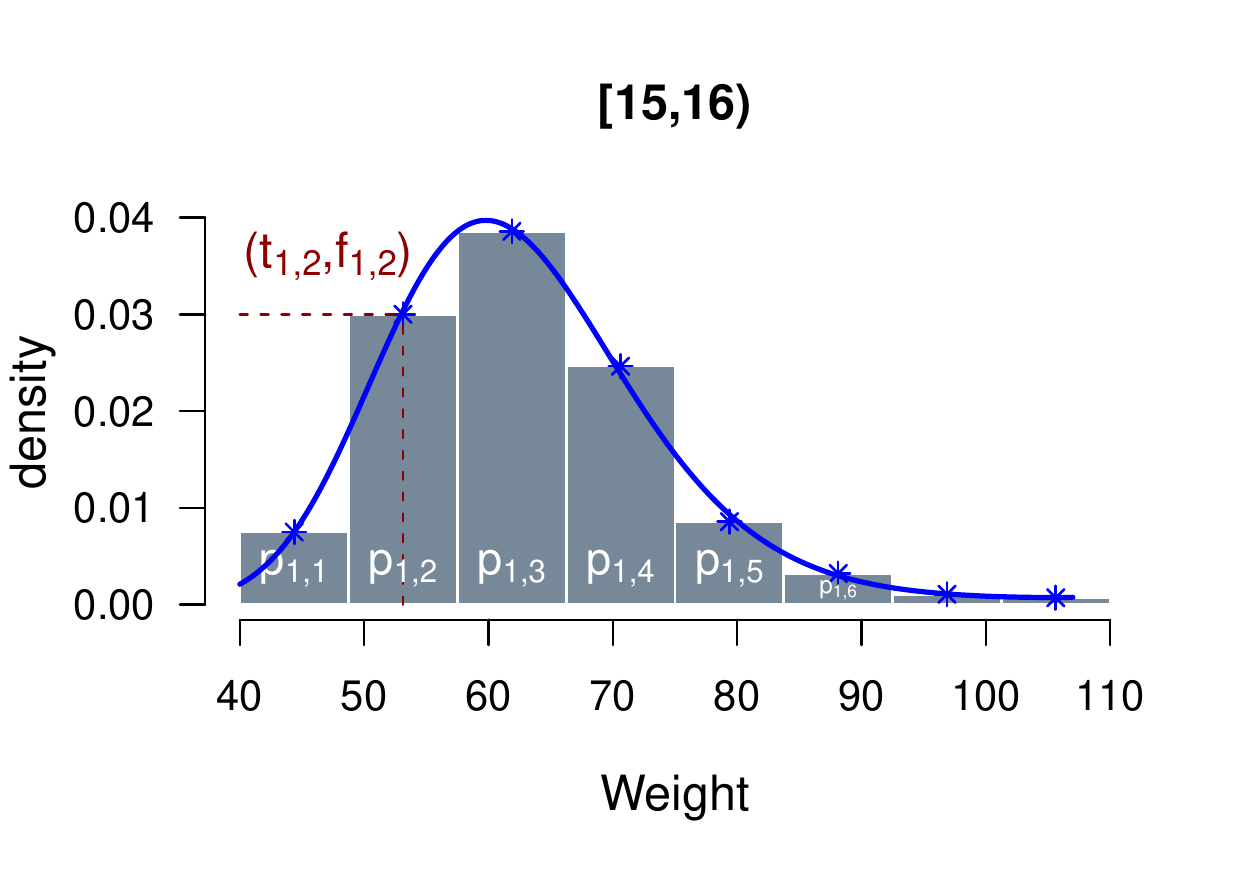}\includegraphics[width=0.5\textwidth]{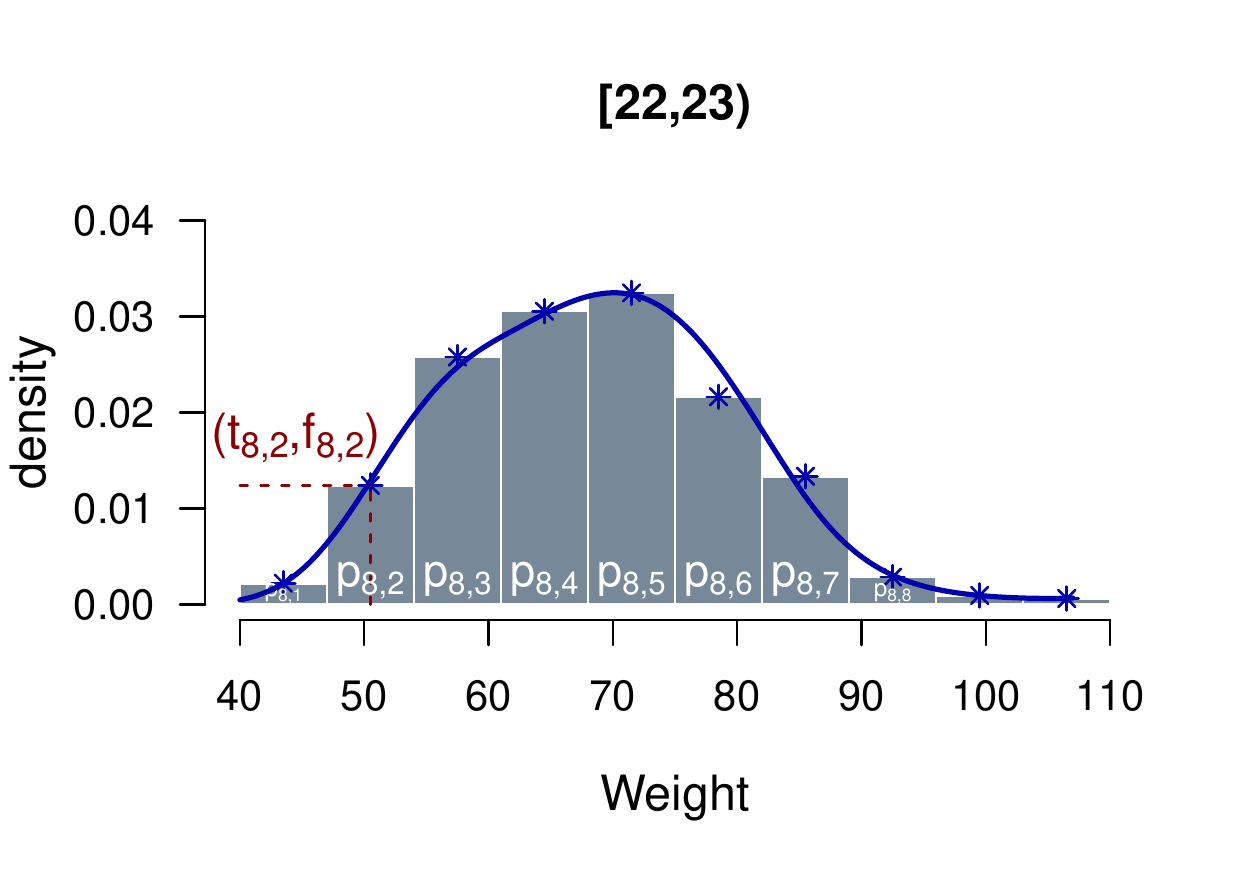} \\
		 \includegraphics[width=0.5\textwidth]{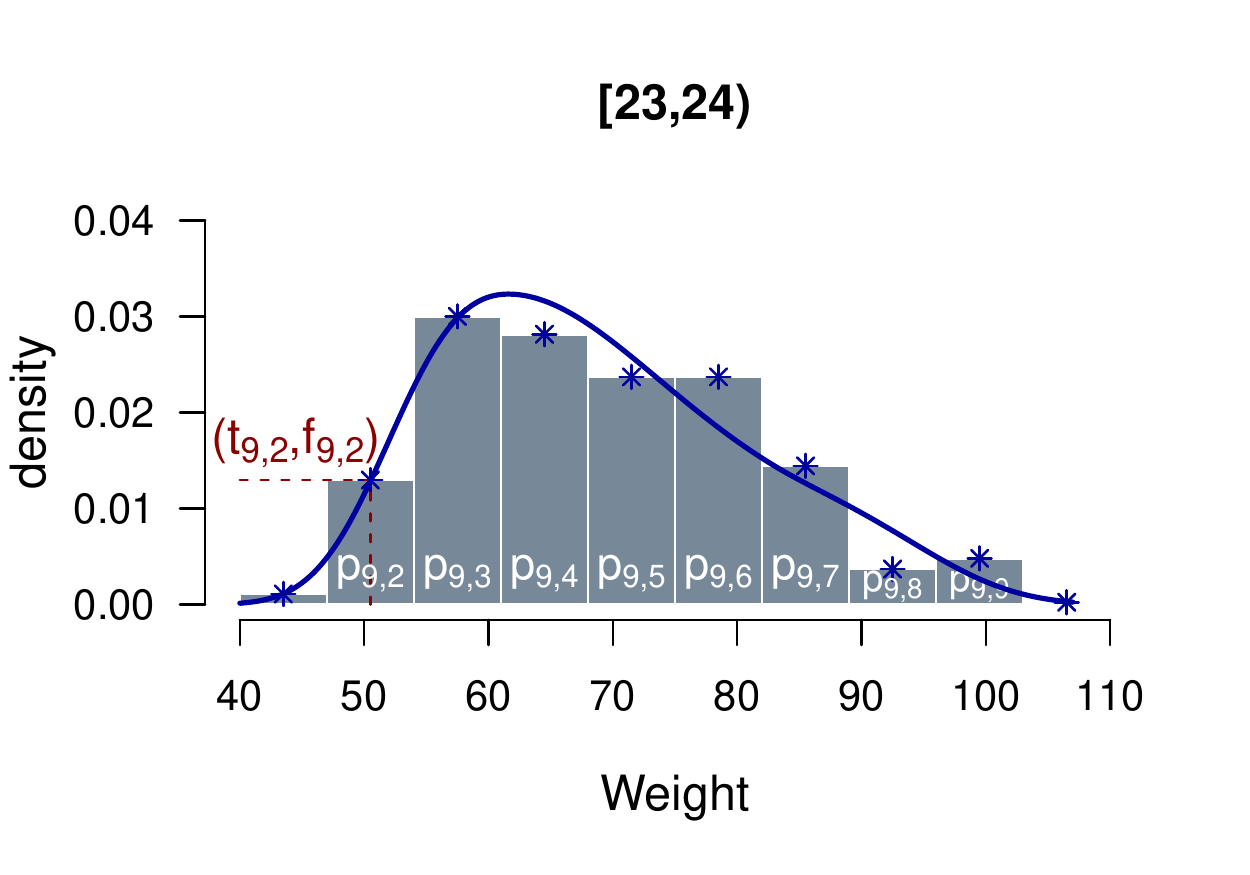}\includegraphics[width=0.5\textwidth]{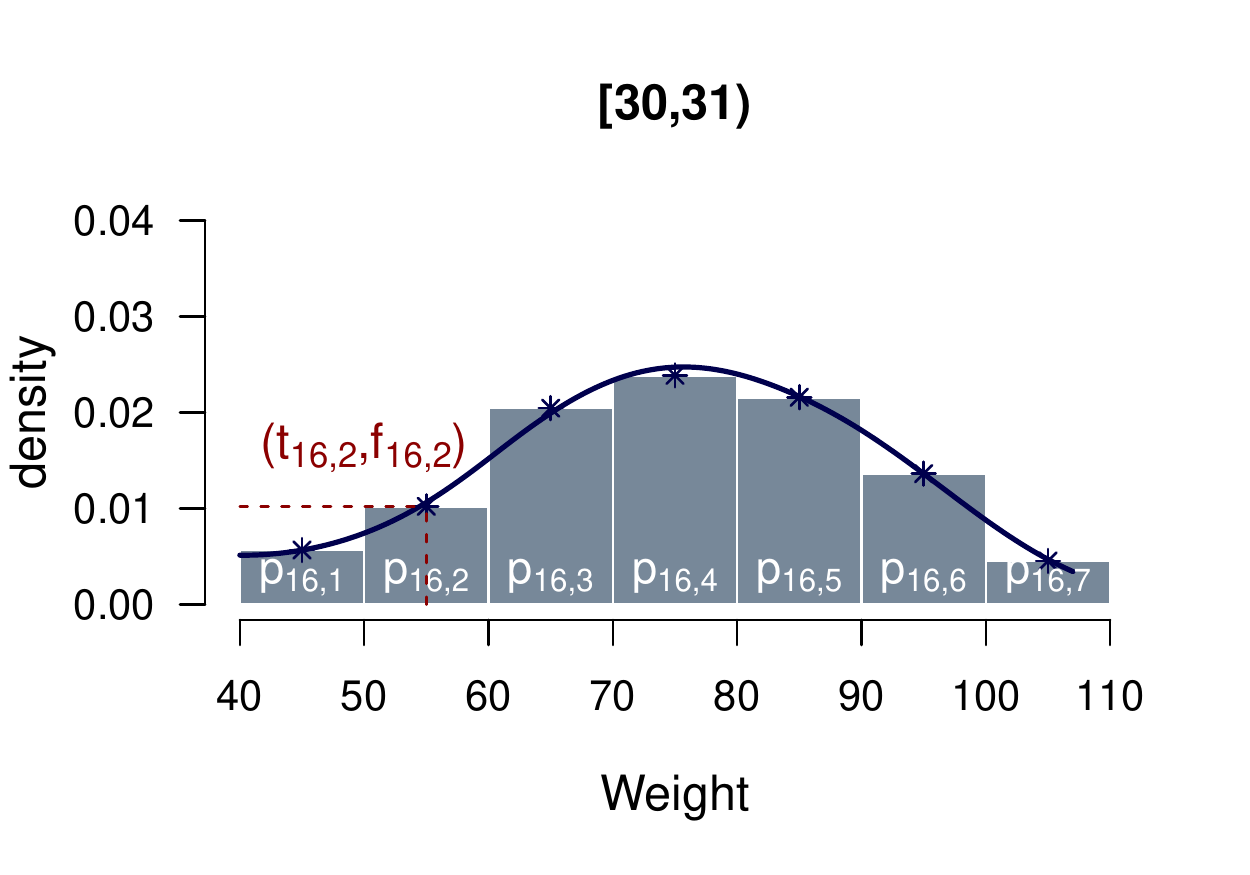}	
\caption{Histograms for four age groups: $\left[15,16\right), \left[22,23\right), \left[23,24\right)$ and $\left[30,31\right)$ together with estimated probability density functions via compositional smoothing splines. Asterisks indicate discrete data $(t_{i,j},f_{i,j}), i=1,8,9,16, j = 1,2,\ldots, q_i$, and $p_{ij},  i=1,8,9,16, j = 1,2,\ldots, q_i$ indicate proportions of equidistant classes resulted for given partition of the range weight body values.}
\label{fig:data_hist}
\end{figure}

Importantly, the quality of smoothing is the same irrespective of the shape of the distribution. It might just happen that the smoothed densities exhibit heavier tails although they are not indicated by the data $f_{i,j}$, even with some artifacts typical for overfitting. This is obviously due to keeping the zero integral constraint, which is more sensitive to deviations from monotonic character of densities at their tails. A possible way out, applied here, was to reduce slightly the range from $I = \left[40,110\right]$ to $I = \left[40,107\right]$ in order to keep predominantly the monotonic behavior of (normalized) counts $f_{i,j}$ at the right tails.

\begin{figure}[t!]
    \centering
		\begin{subfigure}[b]{1\textwidth}
        \centering
        \includegraphics[width=0.45\textwidth]{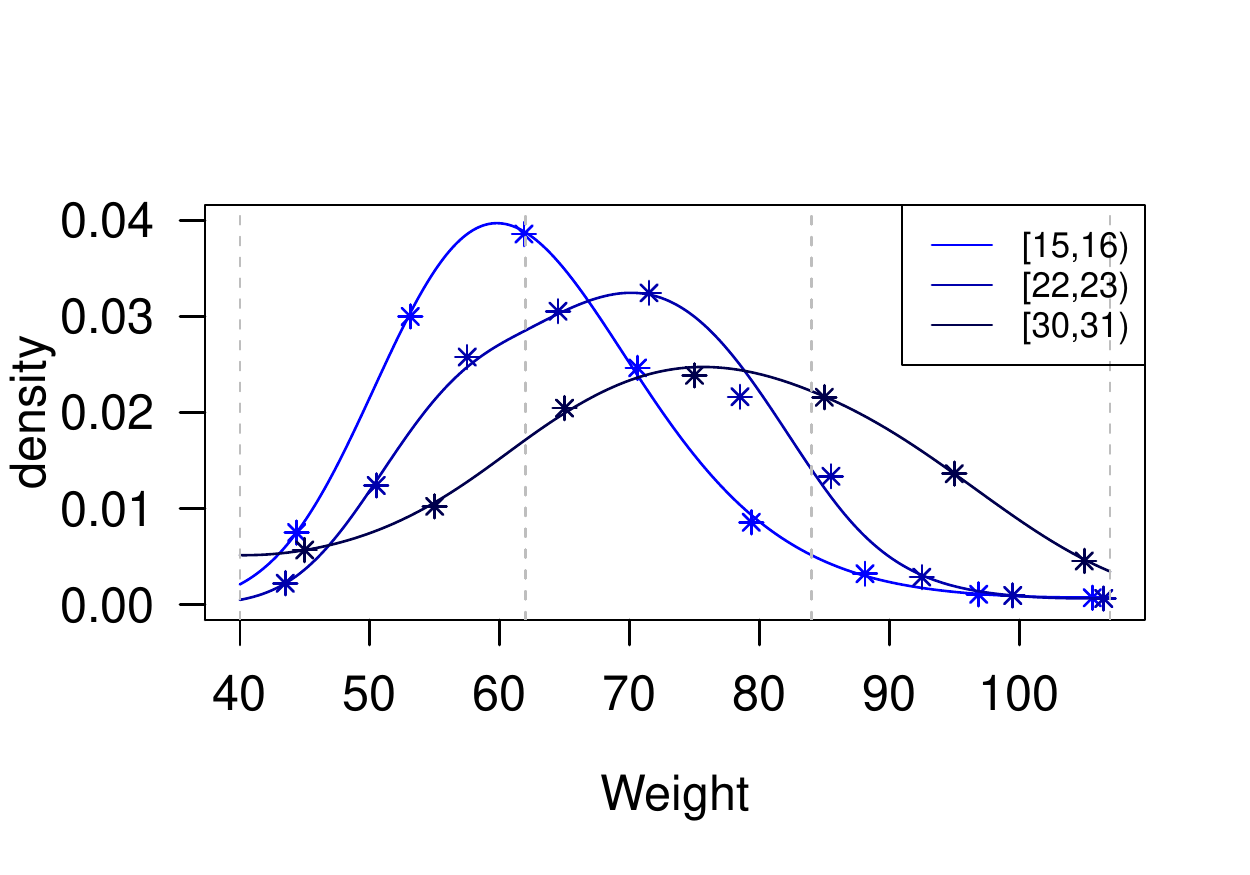}\includegraphics[width=0.45\textwidth]{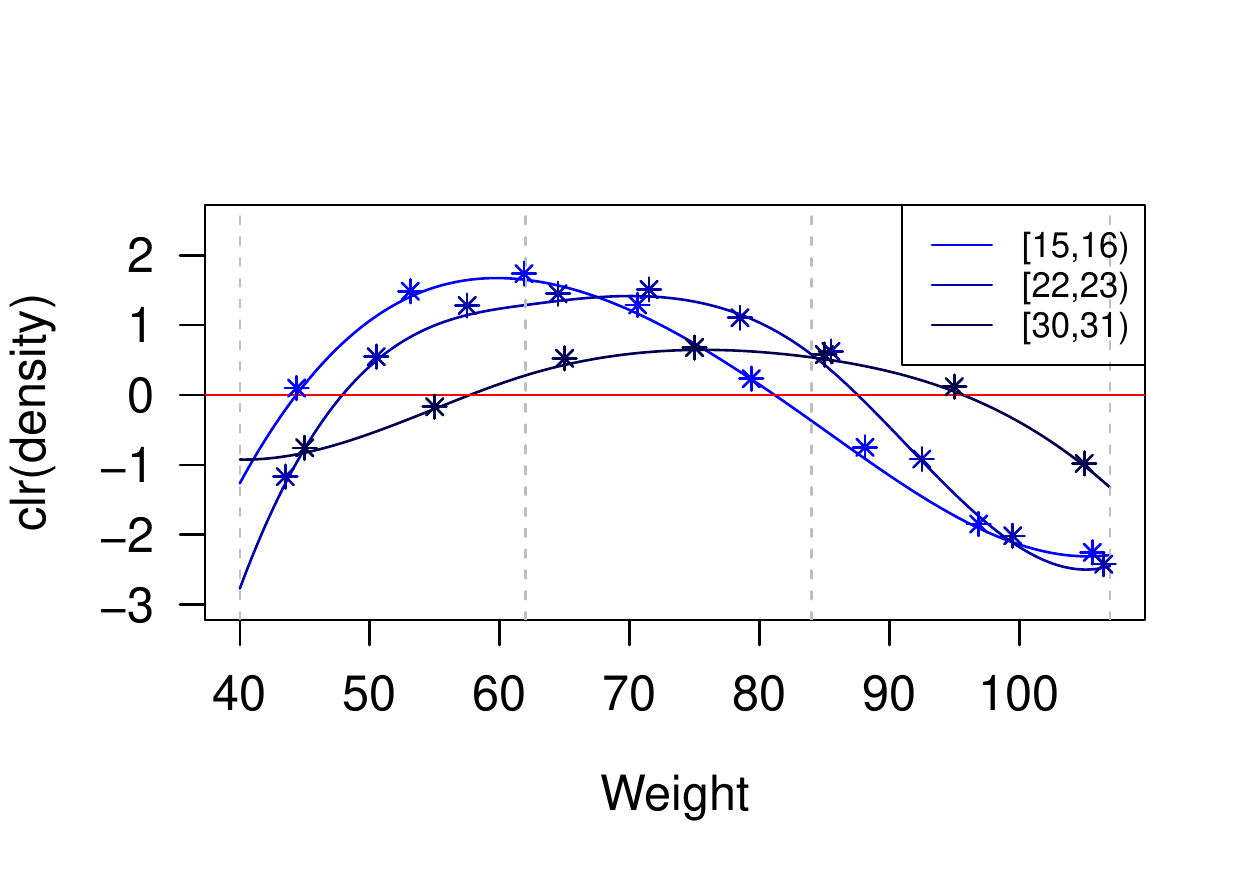}
        \caption{Example of smoothed raw density data in $\mathcal{B}^2$ (left) and $L^2_0$ (right) spaces.}
				\label{fig:data_spliny_example}
		\end{subfigure}
    \begin{subfigure}[b]{1\textwidth}
        \centering
				\includegraphics[width=0.45\textwidth]{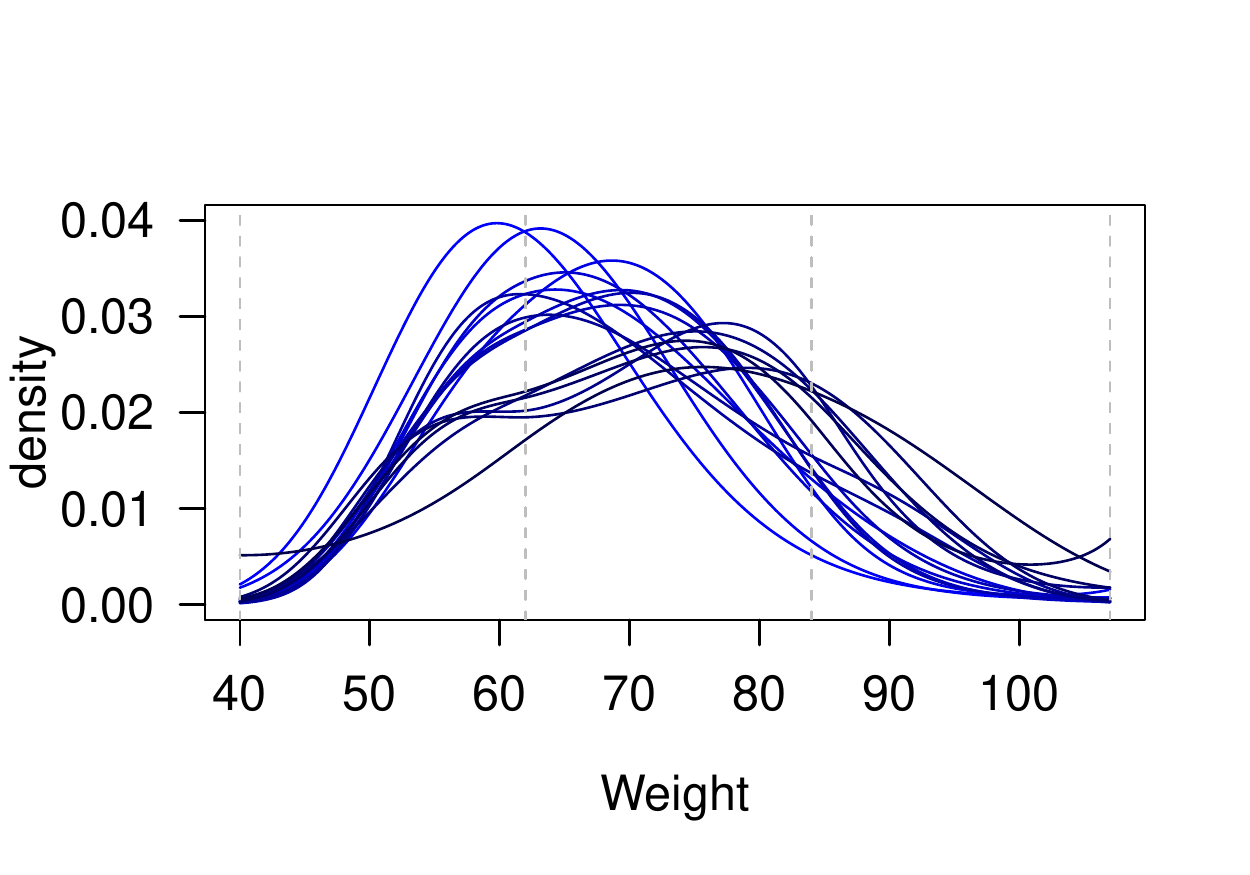}\includegraphics[width=0.45\textwidth]{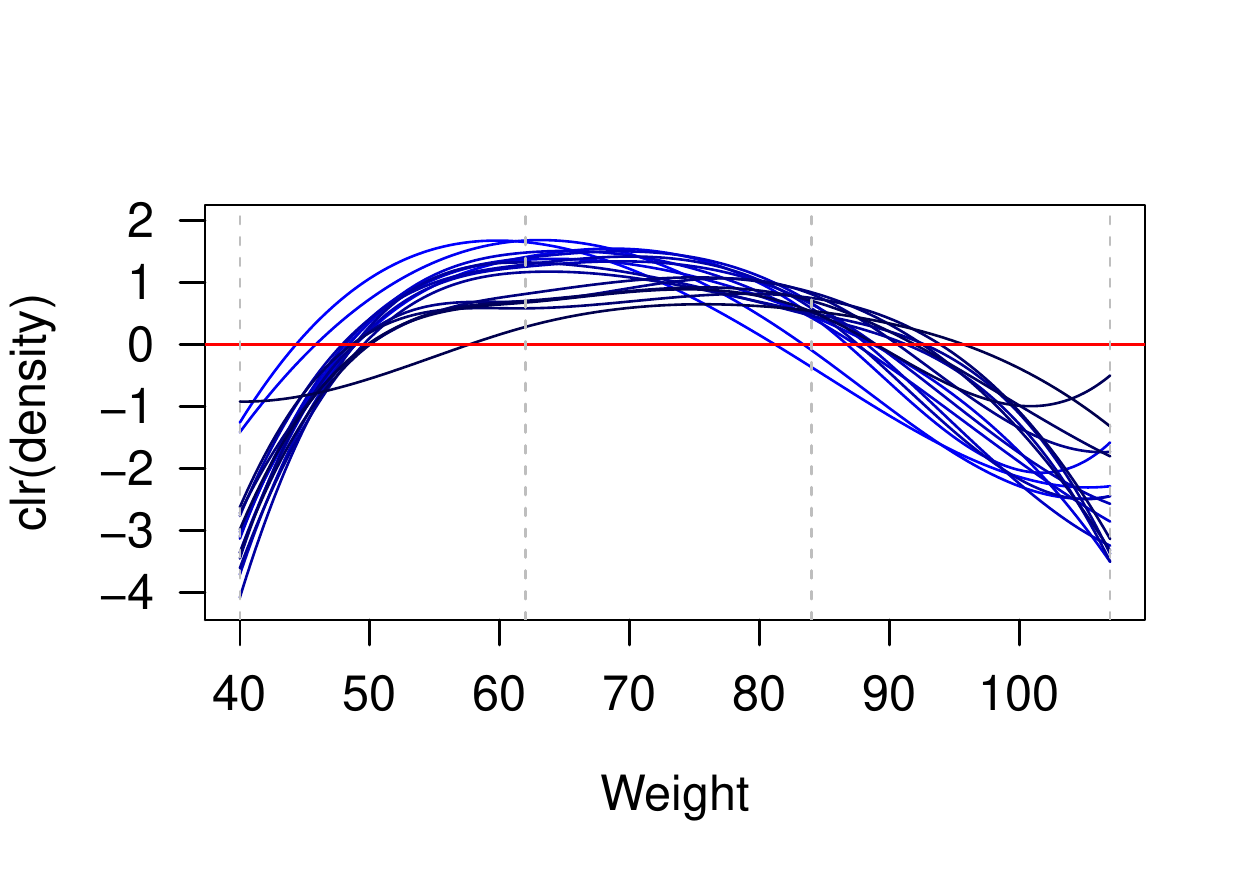}
        \caption{Smoothed raw density data in $\mathcal{B}^2$ (left) and $L^2_0$ (right) spaces.}
				\label{fig:data_spliny_all}
    \end{subfigure}
		\caption{Smoothed weight density functions via compositional smoothing splines in $\mathcal{B}^2$ space and its clr transformation in $L^2_0$ (right) spaces. Data are displayed on blue scale distinguishing age groups: with increasing age, the intensity of blue color increases. Vertical dashed gray lines indicate knots position.}
\label{fig:data_spliny}
\end{figure}

Of course, the smoothed data can be further analyzed using methods of functional data analysis \cite{ramsay05}, adapted in order to respect specific properties of densities. It is demonstrated here for case of the compositional functional principal component analysis (SFPCA) \cite{hron16}. This statistical tool has been recently designed based on the Bayes space methodology, so it enables to capture the main modes of \textit{relative} variability in a data set consisting of sampled density functions. Given a data set of $N$ zero-mean functional observations $X_1,\ldots,X_N$ in $\mathcal{B}^2(I)$, SFPCA aims to find the (normalized and orthogonal) directions of maximum variability in dataset, i.e., a collection of density functions $\left\{\theta_\kappa\right\}_{\kappa \geq 1}$ called simplicial functional principal components (SPFCs) which maximize the following objective function over $\theta \in \mathcal{B}^2(I)$,
\begin{equation} \label{eq_max}
\sum_{i=1}^N \left\langle X_i, \theta\right\rangle_{\mathcal{B}}^2 \: \text{subject to }
\left\| \theta \right\|_{\mathcal{B}} = 1; \: \text{with} \qquad \left\langle \theta, \theta_\kappa \right\rangle_{\mathcal{B}} =0, \: \kappa<K,
\end{equation}
where the orthogonality condition is assumed to be fulfilled for $K\geq 2$. Since $\left\langle X_i, \theta \right\rangle_{\mathcal{B}}$ represents a projection of $X_i$ along the direction $\theta \in \mathcal{B}^2(I)$, in fact we look for orthogonal basis functions in $\mathcal{B}^2(I)$ maximizing the relative variability of these projections. The maximization task (\ref{eq_max}) is efficiently implemented by applying the clr transformation (\ref{fclr}) and the output can be back-transformed from the $L^2(I)$ to the $\mathcal{B}^2(I)$ space, as detailed in \cite{hron16}.

The interpretation of SFPCs can be performed by displaying:
\begin{itemize}
	\item individual SFPCs (as clr transformed density functions, SFPCs always represent contrasts between the parts of domain $I$);
	\item overall mean density function $\bar{X}$ along with its perturbation by SFPCs powered by a suitable constant,
	\begin{equation}
	\begin{split}
	\bar{X} \oplus  \sqrt{\rho}_\kappa \odot \theta_\kappa, \quad
	\bar{X} \ominus \sqrt{\rho}_\kappa \odot \theta_\kappa,
	\end{split}
	\end{equation}
	where ${\rho}_\kappa$ is an amount of the variability of the dataset along the direction $\textit{x}_\kappa$ and it holds ${\rho}_1 \geq {\rho}_2 \geq \ldots$. This is a natural choice because SFPCs represent variation around the overall mean density function;
	\item  the projection of dataset along the directions $\theta_\kappa$,
	$$
	\left\langle X_i, \theta_\kappa \right\rangle_{\mathcal{B}} \odot \theta_\kappa = \textit{x}_{i\kappa} \odot \theta_\kappa, \quad i=1,\ldots,N,
	$$
	where $\textit{x}_{i\kappa} = \left\langle X_i, \theta_\kappa \right\rangle_{\mathcal{B}}, i=1,\ldots,N$ are so called principal component scores associated with the $\kappa$th SFPCs $\theta_\kappa$.	The scores can be plotted for pairs of the first SFPCs to assess the relationship among sampled density functions or to reveal presence of outlying observations;
	\item to complete the above interpretation it is important to note that the original functions $X_1,\ldots,X_N$ and the eigenfunctions $\theta_{\kappa},\kappa\geq 1$ are based on a $CB$-spline basis expansion 
\begin{equation}
\label{SB-splines-SFPCA}
X_i(\cdot)=\bigoplus_{\nu=-3}^1{c_{i,\nu}\odot\zeta_{\nu}(\cdot)}, \quad
\theta_{\kappa}(\cdot)=\bigoplus_{\nu=-3}^1{b_{\kappa,\nu}\odot\zeta_{\nu}(\cdot)},
\end{equation}
respectively, by considering (\ref{SBspline-weight}).
\end{itemize}

SFPCA is also a statistical method for reducing dimensionality of dataset. The number of SFPCs can be determined from the scree plot which displays cumulative percentage of the total variance explained by each subsequent SFPC. That is, the dimensionality is identified by a point in scree plot at which explained variability drops off.

For the actual computation, $CB$-splines (\ref{SB-splines-SFPCA}) of the input functional observations, represented by the corresponding $ZB$-splines (\ref{SBspline-weight}), were expressed by $B$-splines with $B$-spline coefficients (listed in Table \ref{tab:B-coef}) using formula (\ref{mn}). The output of SFPCA for body weight density functions is reported in Figure \ref{fig:sfpca}. According to scree plot (Figure \ref{fig:sfpca-scor}), two or three SFPCs should be taken, but we resort to use only first two of them which capture together almost 85$\%$ of the total variability of the data set. The first SFPC (Figure \ref{fig:sfpca-har1}) represents the contrast between the weight below and above 78 kg, which could be considered as a reference (average) weight. Hence, higher scores along the SFPC$_1$ are expected for age groups with higher incidence of individuals with higher body weight than the average, and, conversely, lower scores are associated with age groups with higher incidence of individuals with lower body weight then the average. The interpretation of SFPC$_1$ can be obviously linked with age, see Figure \ref{fig:sfpca-scor}. The scree plot more or less separates rather right skewed weight density functions of younger age groups (located on the left in the scree plot) from those more symmetric ones associated with older age groups (located on the right in the scree plot).

The second SFPC (Figure \ref{fig:sfpca-har2}) characterizes the variability within the tails of density functions, i.e. the main contribution to the variability along SFPC$_2$ is provided by the lowest and highest weight values ($\leq$ 51 kg and $\geq$ 98 kg respectively). It contrasts low and high weights (associated with high scores along the SFPC$_2$) against middle weight values (associated with low scores along the SFPC$_2$), see Figure \ref{fig:sfpca-scor}. The consistent interpretation can be also observed from Figure \ref{fig:sfpc-harm12} which displays the variation along the first two directions -- SFPC$_1$ and SFPC$_2$ -- with respect to sample mean $\bar{f}(t), t \in I$ (i.e. $\bar{f} \oplus / \ominus 2\sqrt{\rho_{\kappa}} \odot \text{SFPC}_{\kappa}, {\kappa}=1,2$).

Figures \ref{fig:sfpca-proj1} and \ref{fig:sfpca-proj2}, respectively, represent two main modes of variability in the data set ($\left\langle f_i, \theta_\kappa \right\rangle_{\mathcal{B}} \odot \theta_\kappa, \kappa = 1,2, i=1,\ldots,N$). For instance, the variation along SFPC$_2$ is confirmed to be exhibited in tails of density functions and the observations with lowest (gold curve) and highest score (red curve) further support the conclusions made so far. The high scores along the second direction thus reflect heavier tails and, conversely, the low scores along the second direction reflects low incidence of individuals with \textit{extreme} (both small and high) weights. Nevertheless, the relationship of scores (Figure \ref{fig:sfpca-scor}) is apparent: at the beginning, they continue to fall, reach a bottom and then continue to grow. The relationship might be partially explained by unequal representation of men and women in age groups and unequal number of observations in these age groups. Another reason might be that data corresponding to age groups with low SFPC$_2$ scores were collected mostly from students of the Faculty of Physical Culture at Palack\'y University in Olomouc, Czech Republic, which form more homogeneous population than an average one. In any case, the second SFPC reveals an interesting feature which is worth to be further investigated.

\begin{figure}
    \centering
    \begin{subfigure}[t]{0.4\textwidth}
        \centering
        \includegraphics[width=\textwidth]{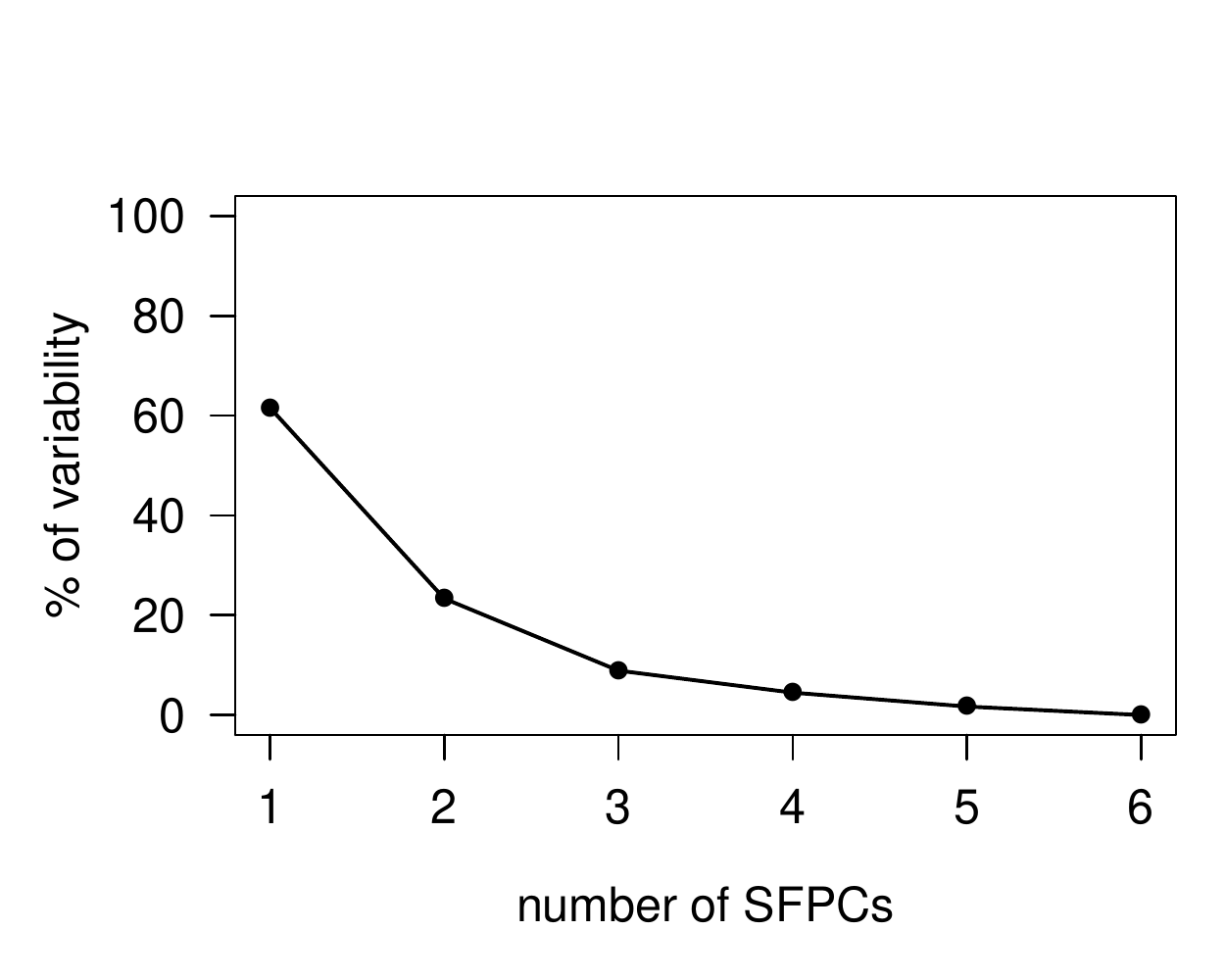}
				\caption{Explained variability.}
					\label{fig:sfpca-scree}
    \end{subfigure}
			\begin{subfigure}[t]{0.4\textwidth}
        \centering
        \includegraphics[width=\textwidth]{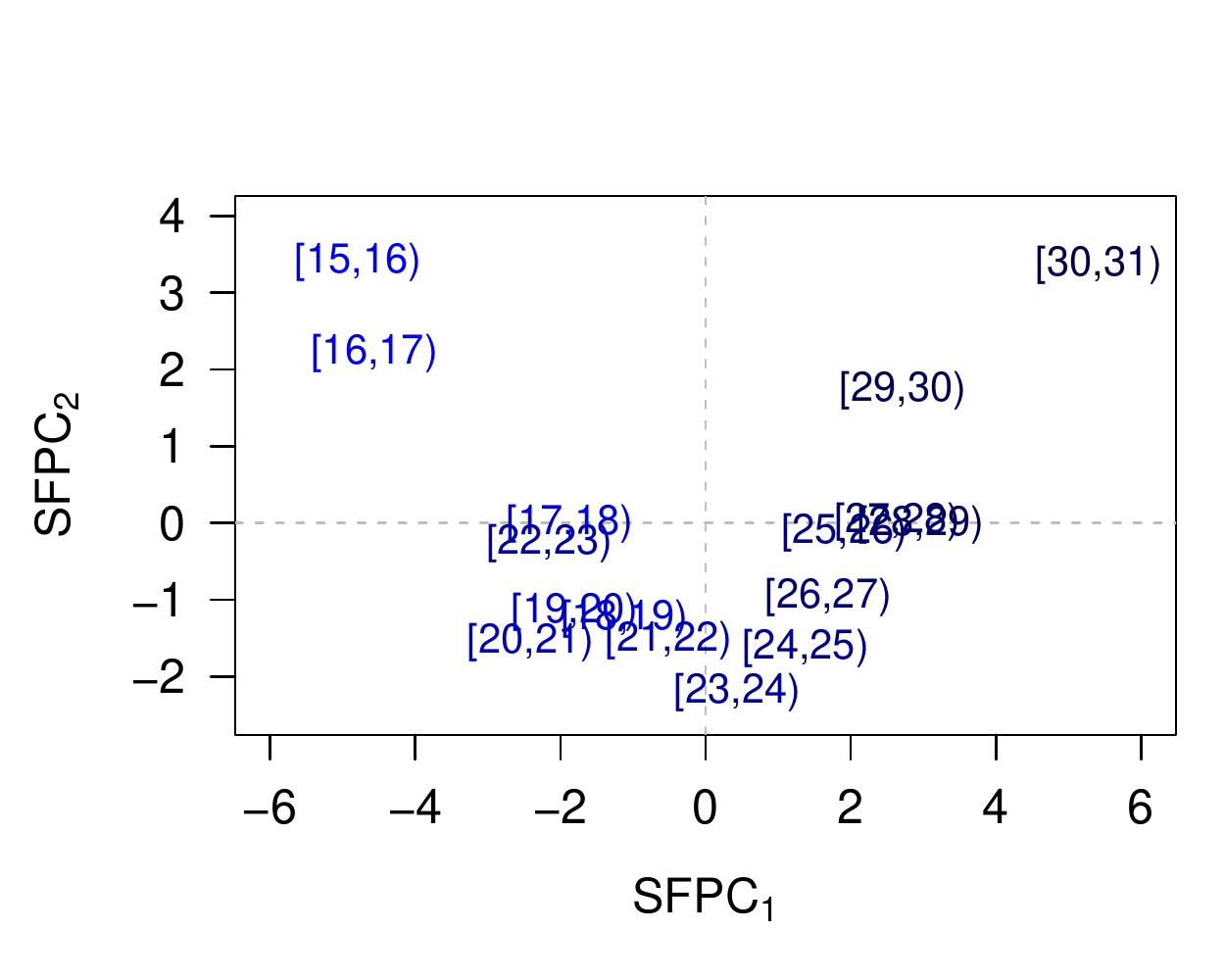}
				\caption{Scores for SFPC$_1$ and SFPC$_2$.}
				\label{fig:sfpca-scor}
	\end{subfigure}
    \centering
	\begin{subfigure}[t]{0.41\textwidth}
        \centering
        \includegraphics[width=\textwidth]{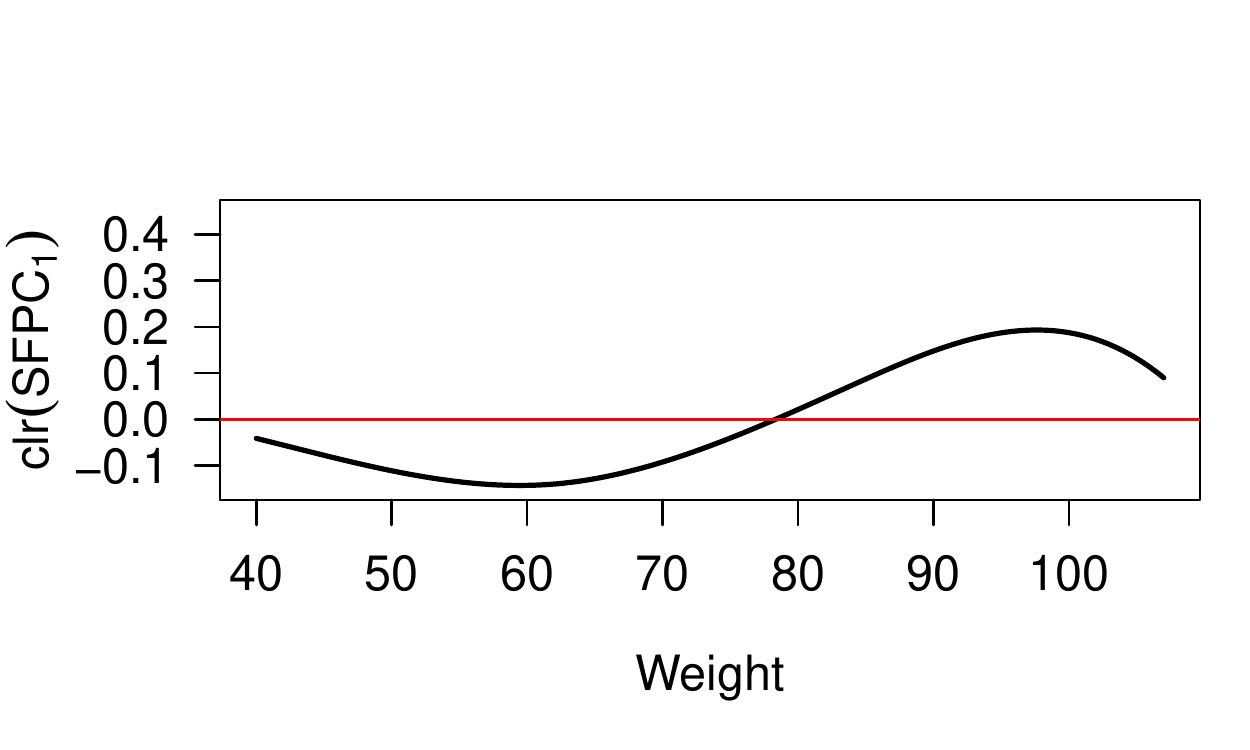}
        \caption{SFPC$_1$ ($61.6\%$ of variability).}
				\label{fig:sfpca-har1}
	\end{subfigure}
		\begin{subfigure}[t]{0.41\textwidth}
        \centering
        \includegraphics[width=\textwidth]{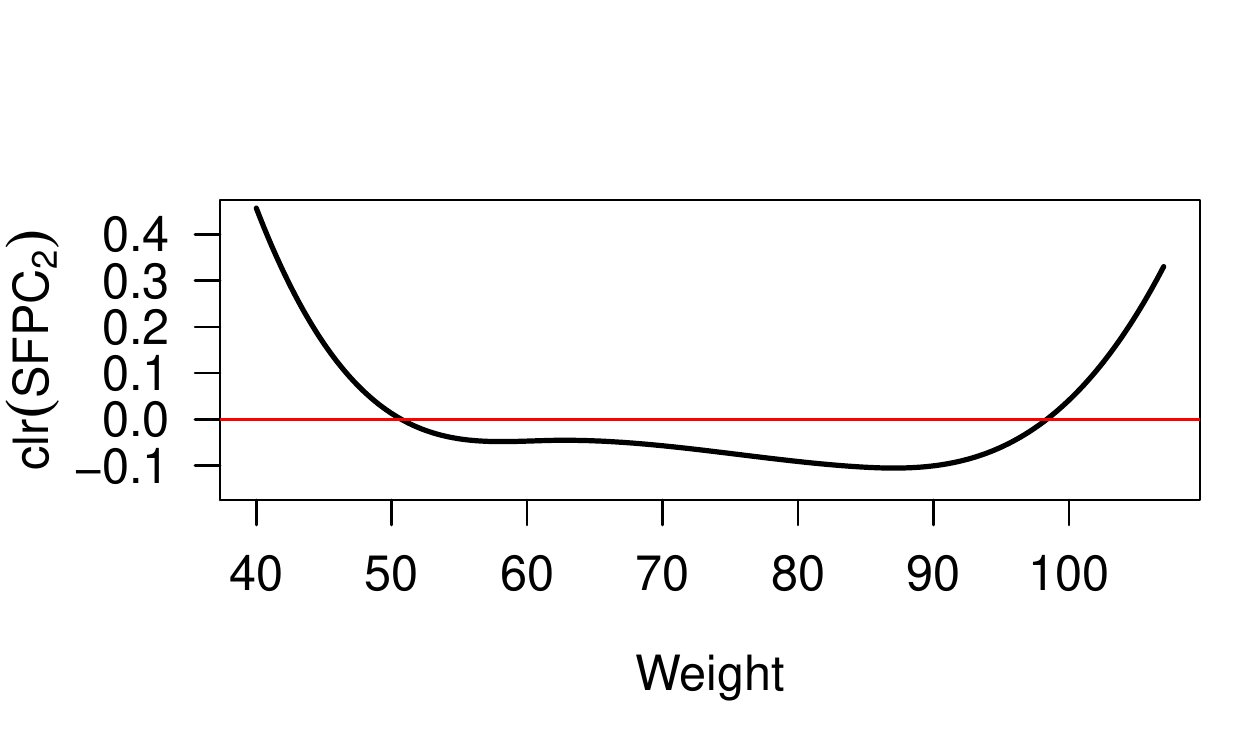}
        \caption{SFPC$_2$ ($23.3\%$ of variability).}
				\label{fig:sfpca-har2}
	\end{subfigure}
    \begin{subfigure}[t]{1\textwidth}
        \centering
        \includegraphics[width=0.41\textwidth]{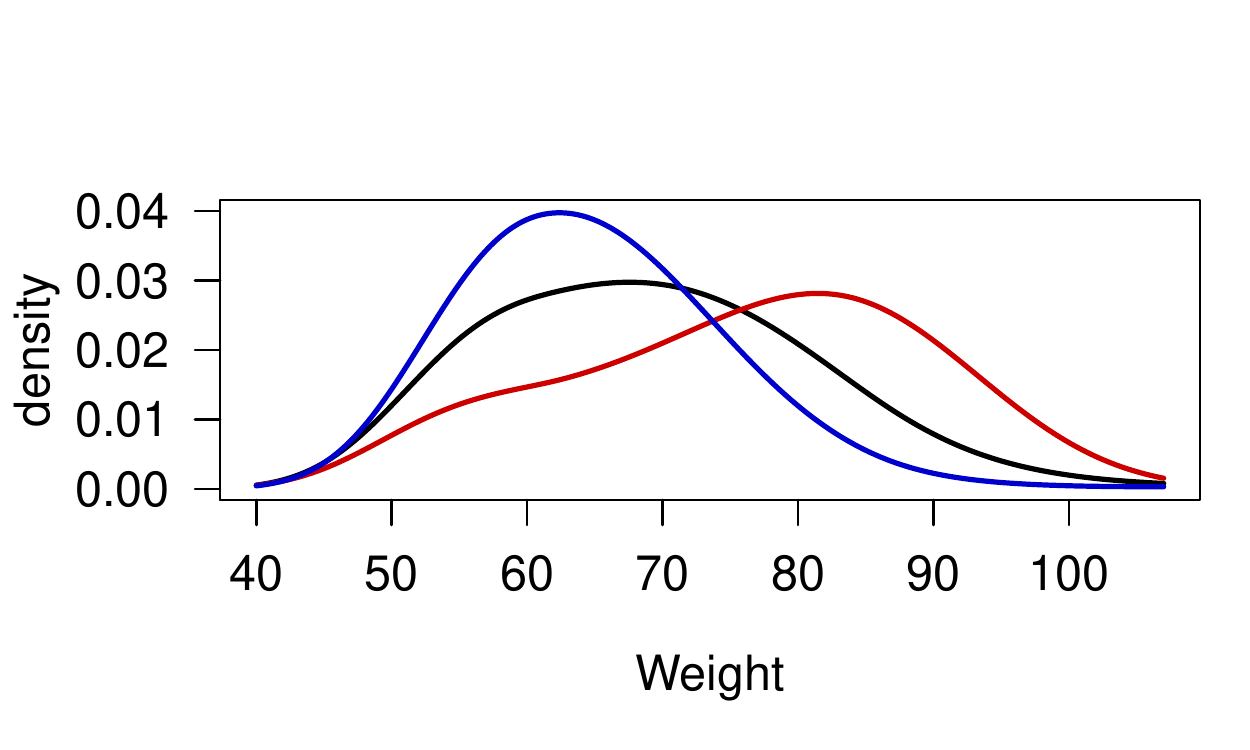}\includegraphics[width=0.41\textwidth]{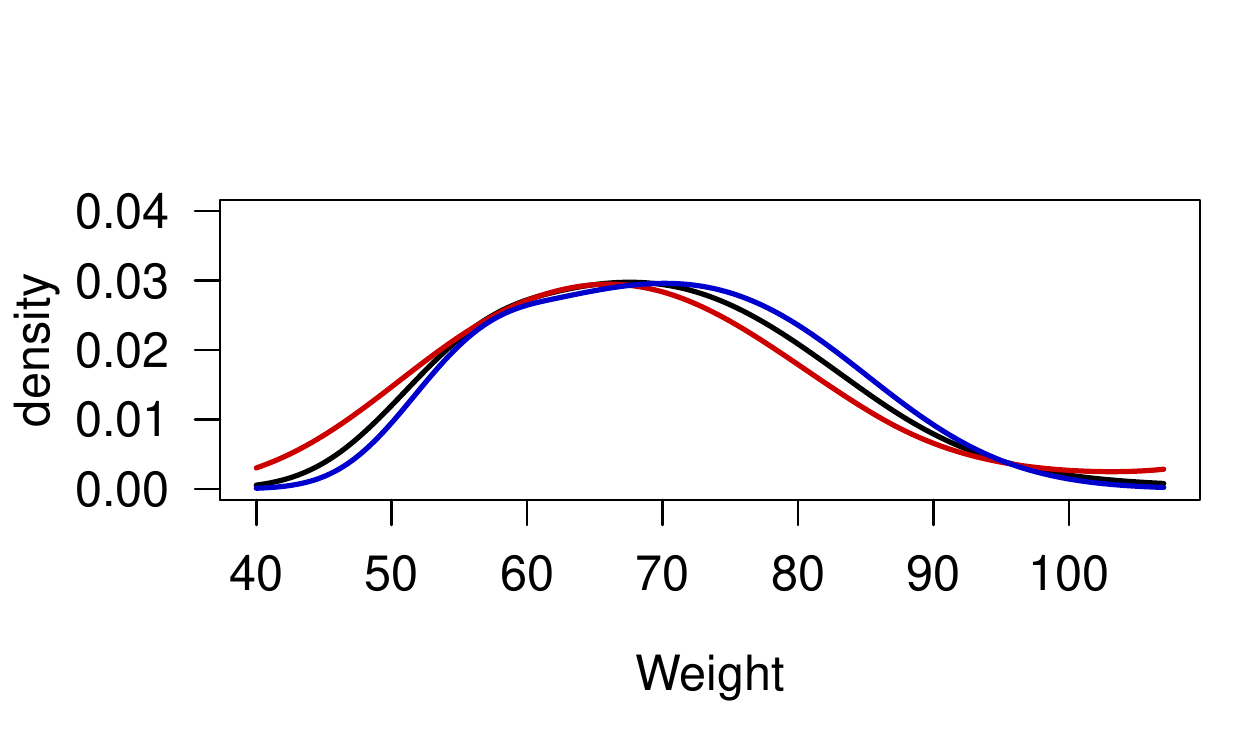}
        \caption{Variability around the mean weight density function ($\bar{f} \oplus / \ominus 2\sqrt{\rho_{\kappa}} \odot \text{SFPC}_{\kappa}, {\kappa}=1,2$).}
				\label{fig:sfpc-harm12}
	\end{subfigure}
		\begin{subfigure}[t]{0.41\textwidth}
                \centering
        \includegraphics[width=\textwidth]{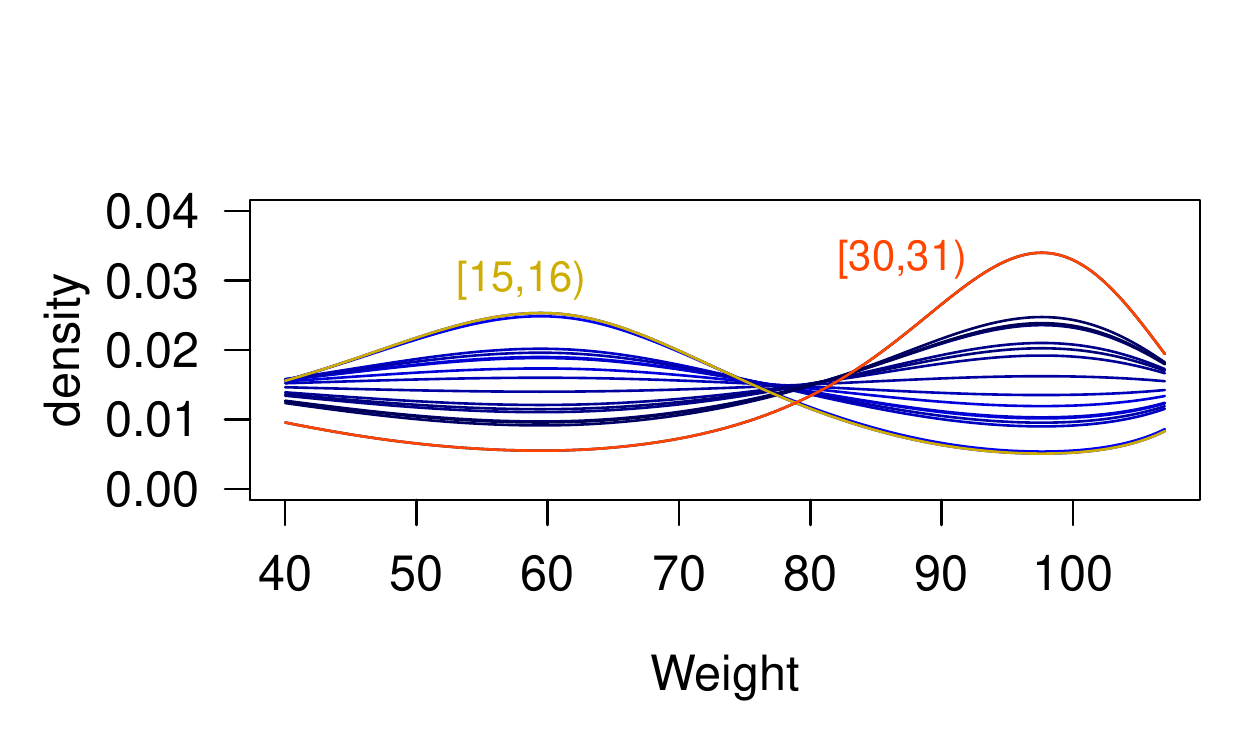}
				\caption{Projection using only SFPC$_1$.}
				\label{fig:sfpca-proj1}
		\end{subfigure}
		\begin{subfigure}[t]{0.41\textwidth}
                \centering
                \includegraphics[width=\textwidth]{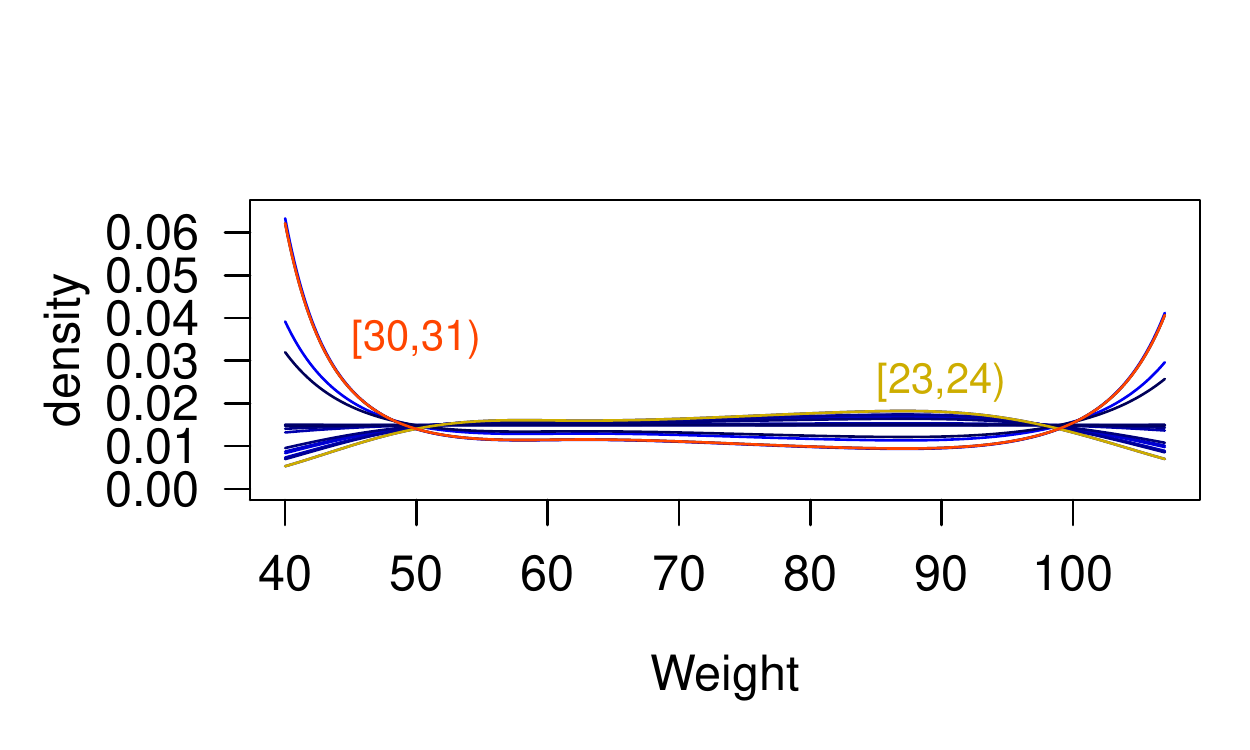}
				\caption{Projection using only SFPC$_2$.}
				\label{fig:sfpca-proj2}
		\end{subfigure}
\caption{SFPCA results for weight density functions. In panel (e), the red curve indicates adding ($\oplus$) of $2\sqrt{\rho_\kappa}$ multiple of SFPC$_\kappa$ and the blue curve indicates subtracting ($\ominus$) of $2\sqrt{\rho_\kappa}$ multiple of SFPC$_\nu$ to the overall mean weight density function $\bar{f}(t)$, $t \in I$, indicated by the black curve (left: $\kappa=1$ and right: $\kappa=2$, respectively); $\rho_{\kappa}$ is the standard deviation along SFPC$_{\kappa}$, ${\kappa}=1,2$.}		
\label{fig:sfpca}
\end{figure}

{\renewcommand{\arraystretch}{0.9}
\begin{table}
\caption{Histogram data for four age groups: $\left[15,16\right), \left[22,23\right), \left[23,24\right)$ and $\left[30,31\right)$. $\textbf{f}_i$ are raw density values at midpoints $\textbf{t}_i=\left(t_{i,1},\ldots,t_{i,q_i}\right)^{\top}$ of weight classes with proportions $\textbf{p}_i$ for $i=1,8,9,16$; $q_i$ indicates the number of the weight classes.}
{\scriptsize
\centering
\begin{tabular}{l || l rrrrrrrrrr | l}
	 \multirow{3}{*}{$\left[15,16\right)$} &
	 $\textbf{p}_{1}$ & 0.0656 &   0.2625 &   0.3375 &   0.2156 &   0.0750 &   0.0281 &   0.0094 &   0.0062 & & &\multirow{3}{*}{$q_1 = 8$}\\
   &$\textbf{f}_{1}$ & 0.0075 &   0.0300 &   0.0386 &   0.0246 &   0.0086 &   0.0032 &   0.0011 &   0.0007 & & &\\
   &$\textbf{t}_{1}$ & 44.375 &   53.125 &   61.875 &   70.625 &   79.375 &   88.125 &   96.875 &  105.625 & & &\\
	 \hline
   \multirow{3}{*}{$\left[22,23\right)$} &
	 $\textbf{p}_{8}$ & 0.0156 &   0.0869 &   0.1804 &   0.2138 &   0.2272 &   0.1514 &   0.0935 &   0.0200 &   0.0067 &   0.0045 & \multirow{3}{*}{$q_8 = 10$}\\
   &$\textbf{f}_{8}$ & 0.0022 &   0.0124 &   0.0258 &   0.0305 &   0.0325 &   0.0216 &   0.0134 &   0.0029 &   0.0010 &   0.0006 &\\
   &$\textbf{t}_{8}$ & 43.5   &   50.5   &     57.5 &     64.5 &     71.5 &     78.5 &     85.5 &     92.5 &     99.5 &    106.5 &\\
	 \hline
	 \multirow{3}{*}{$\left[23,24\right)$} &
	 $\textbf{p}_{9}$ & 0.0078 &   0.0908 &   0.2100 &   0.1971 &   0.1659 &   0.1659 &   0.1011 &   0.0259 &   0.0337 &   0.0017  & \multirow{3}{*}{$q_9 = 10$}\\
   &$\textbf{f}_{9}$ & 0.0011 &   0.0130 &   0.0300 &   0.0282 &   0.0237 &   0.0237 &   0.0144 &   0.0037 &   0.0048 &   0.0002 & \\
   &$\textbf{t}_{9}$ & 43.5   &  50.5    &  57.5    &  64.5    &  71.5    &  78.5    &  85.5    &  92.5    &  99.5    & 106.5    & \\
	 \hline
   \multirow{3}{*}{$\left[30,31\right)$} &
   $\textbf{p}_{16}$ & 0.0568 &   0.1023 &   0.2045 &   0.2386 &   0.2159 &   0.1364 &   0.0455 &&&& \multirow{3}{*}{$q_{16} = 7$}\\
   &$\textbf{f}_{16}$ & 0.0057 &   0.0102 &   0.0205 &   0.0239 &   0.0216 &   0.0136 &   0.0045 &&&&\\
   &$\textbf{t}_{16}$ & 45.0 &  55.0 &  65.0 &  75.0 &  85.0 &  95.0 & 105.0 &&&&\\
\end{tabular}
}
\label{tab:rawdata}
\end{table}
}

{\renewcommand{\arraystretch}{0.8}
\begin{table}
\caption{Input data for smoothing procedure: $\clr(\textbf{f}_{i})$ are raw clr density values at midpoints $\textbf{t}_i=\left(t_{i,1},\ldots,t_{i,q_i}\right)^{\top}$ of the weight classes for $i=1,2,\ldots,N$; $q_i$ indicates the number of the weight classes.
}
{\scriptsize
\centering
\begin{tabular}{l || l rrrrrrrrrr | l}
	 \multirow{2}{*}{$\left[15,16\right)$} &
   $\clr\left[\textbf{f}_{1}\right]$ & 0.100 &    1.486 &     1.737 &   1.289 &      0.233 &  -0.748 &  -1.846 &  -2.252 & & &\multirow{2}{*}{$q_1 = 8$}\\
   &$\textbf{t}_{1}$ & 44.375 &   53.125 &   61.875 &   70.625 &   79.375 &   88.125 &   96.875 &  105.625 & & &\\
	 \hline
	 \multirow{2}{*}{$\left[16,17\right)$} &
   $\clr\left[\textbf{f}_{2}\right]$ & -0.210 &   1.217 &   1.760 &   1.636 &   0.396 &  -0.392 &  -2.001 &  -2.407 & & &\multirow{2}{*}{$q_2 = 8$}\\
   &$\textbf{t}_{2}$ & 44.375 &   53.125 &   61.875 &   70.625 &   79.375 &   88.125 &   96.875 &  105.625 & & &\\
	 \hline
	 \multirow{2}{*}{$\left[17,18\right)$} &
   $\clr\left[\textbf{f}_{3}\right]$ & -1.375 &   0.570 &   1.316 &   1.669 &   1.381 &   0.534 &  -0.364 &  -2.069 &  -1.663 &  &\multirow{2}{*}{$q_3 = 9$}\\
   &$\textbf{t}_{3}$ & 43.889 &  51.667 &  59.444 &  67.222 &  75.000 &  82.778 &  90.556 &  98.333 & 106.111 & & \\
	 \hline
		 \multirow{2}{*}{$\left[18,19\right)$} &
   $\clr\left[\textbf{f}_{4}\right]$ & -1.354 &   0.592 &   1.419 &   1.443 &   1.406 &   1.131 &   0.563 &  -0.661 &  -1.171 &  -3.369  &\multirow{2}{*}{$q_4 = 10$}\\
   &$\textbf{t}_{4}$ & 43.5 &  50.5 &  57.5 &  64.5 &  71.5 &  78.5 &  85.5 &  92.5 &  99.5 & 106.5 &  \\
	 \hline
			\multirow{2}{*}{$\left[19,20\right)$} &
   $\clr\left[\textbf{f}_{5}\right]$ & -1.536 &   0.628 &   1.408 &   1.555 &   1.535 &   1.209 &   0.302 &  -0.774 &  -1.536 &  -2.789  &\multirow{2}{*}{$q_5 = 10$}\\
   &$\textbf{t}_{5}$ & 43.5 &  50.5 &  57.5 &  64.5 &  71.5 &  78.5 &  85.5 &  92.5 &  99.5 & 106.5 &  \\
	 \hline
			\multirow{2}{*}{$\left[20,21\right)$} &
   $\clr\left[\textbf{f}_{6}\right]$ & -1.341 &   0.674 &   1.333 &   1.558 &   1.638 &   1.452 &   0.422 &  -0.568 &  -2.034 &  -3.133  &\multirow{2}{*}{$q_6 = 10$}\\
   &$\textbf{t}_{6}$ & 43.5 &  50.5 &  57.5 &  64.5 &  71.5 &  78.5 &  85.5 &  92.5 &  99.5 & 106.5 &  \\
	 \hline
		 \multirow{2}{*}{$\left[21,22\right)$} &
   $\clr\left[\textbf{f}_{7}\right]$ & -1.746 &   0.451 &   1.185 &   1.463 &   1.411 &   1.131 &   0.531 &  -0.242 &  -1.746 &  -2.439  &\multirow{2}{*}{$q_7 = 10$}\\
   &$\textbf{t}_{7}$ & 43.5 &  50.5 &  57.5 &  64.5 &  71.5 &  78.5 &  85.5 &  92.5 &  99.5 & 106.5 &  \\
	 \hline
		 \multirow{2}{*}{$\left[22,23\right)$} &
   $\clr\left[\textbf{f}_{8}\right]$ & -1.168 &   0.550 &   1.281 &   1.450 &   1.511 &   1.106 &   0.624 &  -0.917 &  -2.015 &  -2.421  &\multirow{2}{*}{$q_8 = 10$}\\
   &$\textbf{t}_{8}$ & 43.5 &  50.5 &  57.5 &  64.5 &  71.5 &  78.5 &  85.5 &  92.5 &  99.5 & 106.5 &  \\
	 \hline
			 \multirow{2}{*}{$\left[23,24\right)$} &
   $\clr\left[\textbf{f}_{9}\right]$ & -1.884 &   0.573 &   1.412 &   1.348 &   1.177 &   1.177 &   0.681 &  -0.680 &  -0.417 &  -3.388  &\multirow{2}{*}{$q_9 = 10$}\\
   &$\textbf{t}_{9}$ & 43.5 &  50.5 &  57.5 &  64.5 &  71.5 &  78.5 &  85.5 &  92.5 &  99.5 & 106.5 &  \\
	 \hline
		\multirow{2}{*}{$\left[24,25\right)$} &
   $\clr\left[\textbf{f}_{10}\right]$ & -1.602 &   0.595 &   1.186 &   1.274 &   1.106 &   0.796 &   0.056 &  -0.423 &  -2.988 &  &\multirow{2}{*}{$q_{10} = 9$}\\
   &$\textbf{t}_{10}$ & 43.889 &  51.667 &  59.444 &  67.222 &  75.000 &  82.778 &  90.556 &  98.333 & 106.111 & & \\
	 \hline
		\multirow{2}{*}{$\left[25,26\right)$} &
   $\clr\left[\textbf{f}_{11}\right]$ &  -1.401 &   0.471 &   0.768 &   0.824 &   1.145 &   0.850 &   0.209 &  -1.178 &  -1.688 &  &\multirow{2}{*}{$q_{11} = 9$}\\
   &$\textbf{t}_{11}$ & 43.889 &  51.667 &  59.444 &  67.222 &  75.000 &  82.778 &  90.556 &  98.333 & 106.111 & & \\
	 \hline
		 \multirow{2}{*}{$\left[26,27\right)$} &
   $\clr\left[\textbf{f}_{12}\right]$ & -1.045 &   0.513 &   0.901 &   1.180 &   1.258 &   0.513 &  -0.485 &  -2.836 & & &\multirow{2}{*}{$q_{12} = 8$}\\
   &$\textbf{t}_{12}$ & 44.375 &   53.125 &   61.875 &   70.625 &   79.375 &   88.125 &   96.875 &  105.625 & & &\\
	 \hline
			\multirow{2}{*}{$\left[27,28\right)$} &
   $\clr\left[\textbf{f}_{13}\right]$ & -0.816 &   0.570 &   0.742 &   0.742 &   1.056 &   0.570 &  -0.256 &  -2.608 & & &\multirow{2}{*}{$q_{13} = 8$}\\
   &$\textbf{t}_{13}$ & 44.375 &   53.125 &   61.875 &   70.625 &   79.375 &   88.125 &   96.875 &  105.625 & & &\\
	 \hline
			\multirow{2}{*}{$\left[28,29\right)$} &
   $\clr\left[\textbf{f}_{14}\right]$ & -1.155 &   0.579 &   0.790 &   0.965 &   0.690 &  -0.308 &  -1.561 & & & &\multirow{2}{*}{$q_{14} = 7$}\\
   &$\textbf{t}_{14}$ & 45.0 &  55.0 &  65.0 &  75.0 &  85.0 &  95.0 & 105.0 & & & &\\
	 \hline
			\multirow{2}{*}{$\left[29,30\right)$} &
   $\clr\left[\textbf{f}_{15}\right]$ & -1.060 &   0.480 &   0.837 &   0.674 &   0.614 &  -0.773 &  -0.773 & & & &\multirow{2}{*}{$q_{15} = 7$}\\
   &$\textbf{t}_{15}$ & 45.0 &  55.0 &  65.0 &  75.0 &  85.0 &  95.0 & 105.0 & & & &\\
	 \hline
			\multirow{2}{*}{$\left[30,31\right)$} &
   $\clr\left[\textbf{f}_{16}\right]$ & --0.756 &  -0.168 &   0.525 &   0.679 &   0.579 &   0.120 &  -0.979 & & & &\multirow{2}{*}{$q_{16} = 7$}\\
   &$\textbf{t}_{16}$ & 45.0 &  55.0 &  65.0 &  75.0 &  85.0 &  95.0 & 105.0 & & & &\\
\end{tabular}
}
\label{tab:clrdata}
\end{table}
}

{\renewcommand{\arraystretch}{0.9}
\begin{table}
\caption{$ZB$-spline coefficients for clr transformed density functions of $N=16$ age groups.}
\centering
{\footnotesize
\begin{tabular}{l||R{1.3cm}R{1.3cm}R{1.3cm}R{1.3cm}R{1.3cm}}
	age group &  \multicolumn{5}{ c }{ spline coefficients, $\; \textbf{z}_i = \left(z_{i,-3}, \ldots,z_{i,1}\right)^{\top}, i=1,\ldots,N$}\\
\hline
 $\left[15,16\right)$ &  -6.950 &   6.647 &  46.536 &  40.973 &  13.163 \\
 $\left[16,17\right)$ &  -7.806 &  -0.596 &  41.616 &  45.181 &  14.083 \\
 $\left[17,18\right)$ & -16.677 & -11.292 &  18.284 &  43.917 &   9.102 \\
 $\left[18,19\right)$ & -17.067 &  -8.988 &  21.373 &  33.533 &  20.188 \\
 $\left[19,20\right)$ & -18.483 &  -9.902 &  22.408 &  38.249 &  16.447 \\
 $\left[20,21\right)$ & -17.242 &  -7.010 &  18.199 &  46.788 &  18.682 \\
 $\left[21,22\right)$ & -20.452 & -10.875 &  11.653 &  36.887 &  14.797 \\
 $\left[22,23\right)$ & -15.236 &  -5.368 &  16.735 &  46.421 &  14.071 \\
 $\left[23,24\right)$ & -22.485 & -12.348 &  17.033 &  23.450 &  20.153 \\
 $\left[24,25\right)$ & -19.873 & -14.176 &  13.567 &  20.115 &  19.448 \\
 $\left[25,26\right)$ & -19.011 &  -5.949 &  -4.623 &  30.860 &   9.973 \\
 $\left[26,27\right)$ & -14.997 & -10.545 &   2.638 &  28.225 &  19.143 \\
 $\left[27,28\right)$ & -14.461 &  -4.455 &  -0.689 &  21.892 &  18.070 \\
 $\left[28,29\right)$ & -18.518 & -11.045 &  -2.723 &  21.744 &  10.395 \\
 $\left[29,30\right)$ & -16.445 &  -9.417 &  -1.814 &  23.562 &   2.889 \\
 $\left[30,31\right)$ &  -5.077 & -15.534 &  -4.171 &   8.220 &   7.618 \\
\end{tabular}
}
\label{tab:ZB-coef}
\end{table}
}

{\renewcommand{\arraystretch}{0.9}
\begin{table}
\caption{$B$-spline coefficients for clr transformed density functions of $N=16$ age groups.}
\centering
{\footnotesize
\begin{tabular}{l||R{1cm}R{1cm}R{1cm}R{1cm}R{1cm}R{1cm}}
	age group &  \multicolumn{6}{ c }{spline coefficients, $\; \textbf{b}_i= \left(b_{i,-3}, \ldots, b_{i,2}\right)^{\top}, i=1,\ldots,N$} \\
\hline
$\left[15,16\right)$ &  -1.264 &  1.236 &  2.381 & -0.332 & -2.472 & -2.289 \\
$\left[16,17\right)$ &  -1.419 &  0.655 &  2.520 &  0.213 & -2.764 & -2.449 \\
$\left[17,18\right)$ &  -3.032 &  0.490 &  1.766 &  1.530 & -3.095 & -1.583 \\
$\left[18,19\right)$ &  -3.103 &  0.734 &  1.813 &  0.726 & -1.186 & -3.511 \\
$\left[19,20\right)$ &  -3.361 &  0.780 &  1.929 &  0.946 & -1.938 & -2.860 \\
$\left[20,21\right)$ &  -3.135 &  0.930 &  1.505 &  1.707 & -2.498 & -3.249 \\
$\left[21,22\right)$ &  -3.719 &  0.871 &  1.345 &  1.507 & -1.964 & -2.573 \\
$\left[22,23\right)$ &  -2.770 &  0.897 &  1.320 &  1.772 & -2.876 & -2.447 \\
$\left[23,24\right)$ &  -4.088 &  0.922 &  1.754 &  0.383 & -0.293 & -3.505 \\
$\left[24,25\right)$ &  -3.613 &  0.518 &  1.656 &  0.391 & -0.059 & -3.382 \\
$\left[25,26\right)$ &  -3.456 &  1.187 &  0.079 &  2.118 & -1.857 & -1.734 \\
$\left[26,27\right)$ &  -2.727 &  0.405 &  0.787 &  1.528 & -0.807 & -3.329 \\
$\left[27,28\right)$ &  -2.629 &  0.910 &  0.225 &  1.348 & -0.340 & -3.143 \\
$\left[28,29\right)$ &  -3.367 &  0.679 &  0.497 &  1.461 & -1.009 & -1.808 \\
$\left[29,30\right)$ &  -2.990 &  0.639 &  0.454 &  1.515 & -1.838 & -0.502 \\
$\left[30,31\right)$ &  -0.923 & -0.951 &  0.678 &  0.740 & -0.053 & -1.325 \\
\end{tabular}
}
\label{tab:B-coef}
\end{table}
}

\section{Conclusions}

The compositional splines, which enable to construct a spline basis in the clr space of density functions 
($ZB$-spline basis) 
and consequently also in the original space of densities 
($CB$-spline basis), might become an important contribution within the Bayes space methodology for processing of functional data carrying relative information. They provide a solid theoretical base for further developments of the approximation theory in context of the Bayes spaces, but even more importantly, compositional splines can be used also for adaptation of popular methods of functional data analysis for density functions. Here the case of compositional functional principal component analysis was presented, but similarly, e.g., regression analysis or classification methods could be developed. Also further tuning of the compositional splines is possible, here represented by the smoothing compositional splines or by orthonormalization of the $ZB$-basis. The latter case be used for an orthogonal projection of a density function on a subset of $CB$-splines, to further applications within the approximation theory or also for development of the theoretical framework of functional data analysis.

The pending challenge is to generalize the methodology introduced above also to $p$ dimensional density functions, $p>1$, which can be formally extended from any univariate density $f(x),\,x\in I=[a,b]$ to $f(\mathbf{x})$, where $\mathbf{x}=(x_1,\ldots,x_p)^{\top}\in \mathcal{I}=I_1\times\ldots\times I_p=[a_1,b_1]\times\ldots\times[a_p,b_p]$, in Equations (\ref{operations}) to (\ref{bexpansion}); $\eta=b-a$ would be replaced by $H=\prod_{i=1}^p(b_i-a_i)$. Currently an approach which focuses on keeping the zero integral constraint of the clr transformed densities was developed in \cite{guean18} as a generalization of \cite{Mach}, which, however, does not lead to a compositional counterpart of the $B$-spline basis. A consistent approach in this direction is currently under development.

\section*{Acknowledgements}

The authors gratefully acknowledge both the support by Czech Science Foundation
GA18-09188S, the grant IGA\_PrF\_2018\_024 Mathematical Models of the
Internal Grant Agency of the Palack\'y University in Olomouc, and the grant COST
Action CRoNoS IC1408.

\end{document}